\documentclass{amsart}
\usepackage{amsmath}%
\usepackage{amsthm}
\usepackage{amsfonts}%
\usepackage{amssymb}%
\usepackage{graphicx}
\usepackage{tikz-cd}
\usetikzlibrary{cd}

\usepackage{mathrsfs}

\DeclareFontFamily{OT1}{pzc}{}
\DeclareFontShape{OT1}{pzc}{m}{it}{<-> s * [1.10] pzcmi7t}{}
\DeclareMathAlphabet{\mathpzc}{OT1}{pzc}{m}{it}

\newtheorem{theorem}{Theorem}[section]

\newtheorem{corollary}[theorem]{Corollary}

\newtheorem{definition}[theorem]{Definition}
\newtheorem{example}[theorem]{Example}

\newtheorem{lemma}[theorem]{Lemma}

\newtheorem{proposition}[theorem]{Proposition}
\newtheorem{remark}[theorem]{Remark}

\numberwithin{equation}{section}
\def\Xint#1{\mathchoice
{\XXint\displaystyle\textstyle{#1}}%
{\XXint\textstyle\scriptstyle{#1}}%
{\XXint\scriptstyle\scriptscriptstyle{#1}}%
{\XXint\scriptscriptstyle\scriptscriptstyle{#1}}%
\!\int}
\def\XXint#1#2#3{{\setbox0=\hbox{$#1{#2#3}{\int}$}
\vcenter{\hbox{$#2#3$}}\kern-.5\wd0}}

\def\dashint{\Xint-}

\newcommand{\R}{\mathbb{R}}
\newcommand{\Q}{\mathbb{Q}}
\newcommand{\N}{\mathbb{N}}

\newcommand{\D}{\mathbb{D}}

\newcommand{\M}{\mathcal{M}}

\newcommand{\Ha}{\mathcal{H}}
\newcommand{\leb}{\mathcal{L}}

\newcommand{\Ne}[2]{N^{1,#1}(#2)}

\newcommand{\Nel}[2]{N^{1,#1}_{loc}(#2)}
\newcommand{\Nem}[3]{N^{1,#1}(#2;#3)}
\newcommand{\Neml}[3]{N^{1,#1}_{loc}(#2;#3)}

\newcommand{\essinf}{\operatorname{essinf}}

\newcommand{\Mod}{\operatorname{Mod}}
\newcommand{\dist}{\operatorname{dist}}
\newcommand{\diam}{\operatorname{diam}}

\newcommand{\ud}{\mathrm {d}}
\newcommand{\id}{\mathrm {id}}

\newcommand{\inv}{^{-1}}

\newcommand{\md}{\operatorname{md}}
\newcommand{\apmd}{\operatorname{ap md}}

\newcommand{\loc}{\mathrm{loc}}

\title{Maximal metric surfaces and the Sobolev-to-Lipschitz property}

\thanks{The first author was partially supported by the DFG grant SPP 2026. The second author was partially supported by the Vilho, Yrj\"o ja Kalle V\"aisal\"a Foundation (postdoc pool) and by the Swiss National Science Foundation Grant 182423}
\author{Paul Creutz}
\address{University of Cologne, Weyertal 86-90, 50931 K{\"o}ln, Germany}
\email{pcreutz@math.uni-koeln.de}
\author{Elefterios Soultanis}
\address{University of Fribourg, Chemin du Musee 23, CH-1700, Fribourg, Switzerland}
\email{elefterios.soultanis@gmail.com}
\date{\today}

\begin{document}

\begin{abstract}
We find maximal representatives within equivalence classes of metric spheres. For Ahlfors regular spheres these are uniquely characterized by satisfying the seemingly unrelated notions of Sobolev-to-Lipschitz property, or volume rigidity. We also apply our construction to solutions of the Plateau problem in metric spaces and obtain a variant of the associated intrinsic disc studied by Lytchak--Wenger, which  satisfies a related maximality condition.
\end{abstract}

\maketitle
\section{Introduction}
\subsection{Main result}
A celebrated result due to Bonk and Kleiner~\cite{bon02} states that an Ahlfors $2$-regular metric sphere~$Z$ is quasisymetrically equivalent to the standard sphere if and only if it is linearly locally connected. Recently, it was shown in~\cite{lyt17} that the quasisymmetric homeomorphism $u_Z:S^2\rightarrow Z$ may be chosen to be of minimal energy~$E^2_+(u_Z)$, and in this case is unique up to conformal diffeomorphism of $S^2$.\par
The map $u_Z$ gives rise to a measurable Finsler structure on $S^2$, defined by the approximate metric differential $\apmd u_Z$; cf. \cite{lyt18} and Section \ref{sec:discs} below. When $Z$ is a smooth Finsler surface, the approximate metric differential carries all the metric information of $Z$. In the present generality, however, $\apmd u_Z$ is defined only almost everywhere and thus does not determine the length of every curve. 
\begin{definition}\label{def:equivalence}
	Let $Y$ and $Z$ be Ahlfors $2$-regular, linearly locally connected metric spheres. We say that $Y$ and $Z$ are \emph{analytically equivalent} if there exist energy minimizing parametrizations $u_Y$ and $u_Z$ such that
	\begin{equation}\label{eq:equivalence}
	\apmd u_Y=\apmd u_Z
	\end{equation}
	almost everywhere.
\end{definition}

By the aformentioned uniqueness result, analytic equivalence defines an equivalence relation on the class of linearly locally connected, Ahlfors $2$-regular spheres. The main result of this paper states that the equivalence class of such a sphere contains a maximal representative, unique up to isometry. 

\begin{theorem}\label{thm:discmain}
Let $Z$ be a linearly locally connected, Ahlfors $2$-regular sphere. Then there is linearly locally connected, Ahlfors $2$-regular sphere~$\widehat{Z}$ which is analytically and bi-Lipschitz equivalent to $Z$ and satisfies the following properties.
\begin{itemize}
\item[(1)]\textbf{Sobolev-to-Lipschitz property.} If $f\in \Ne 2{\widehat Z}$ has weak upper gradient~$1$, then $f$ has a $1$-Lipschitz representative.
\item[(2)]\textbf{Thick geodecity.} For arbitrary measurable subsets $E,F\subset \widehat{Z}$ of positive measure and $C>1$, one has $\Mod_2\Gamma(E,F;C)>0$.
\item[(3)]\textbf{Maximality.} If $Y$ is analytically equivalent to $\widehat{Z}$, there exists a $1$-Lipschitz homeomorphism $f:\widehat{Z}\rightarrow Y$.
\item[(4)]\textbf{Volume rigidity.} If $Y$ is a linearly locally connected, Ahlfors $2$-regular sphere, and $f:Y \rightarrow \widehat Z$ is a $1$-Lipschitz area preserving map which is moreover cell-like, then $f$ is an isometry.
\end{itemize}
Moreover~$\widehat{Z}$ is characterized uniquely, up to isometry, by any of the listed properties.
\end{theorem}
In the setting of Ahlfors regular metric spheres, Theorem \ref{thm:discmain} links the essential metric investigated in~\cite{dec90,dec91, alb18}, the concept of volume rigidity studied under curvature bounds in~\cite{li12,li14,li15}, the Sobolev-to-Lipschitz property arising in the context of RCD spaces in~\cite{gig13,gig14}, and the notion of thick quasiconvexity whose connection to Poincar\'e inequalities is investigated in~\cite{dur12,dur16,dur12'}.\par 
More generally, an Ahlfors $2$-regular disc $Z$ with finite boundary length admits a quasisymmetric parametrization by the standard disc $\overline{\D}$ if and only if it is linearly locally connected. The parametrization may again be chosen to be of minimal Reshetynak energy and is then unique up to conformal diffeomorphism, see~\cite{lyt17}. A suitable variant of Theorem~\ref{thm:discmain} also holds in this setting. In this case in the formulation of volume rigidity one has to assume that $f$ preserves the boundary curves and the restriction $f:\partial Y\rightarrow \partial Z$ is monotone, instead of assuming that the entire map~$f$ is cell-like.\par 
Other parametrization results for metric surfaces have been obtained for example in~\cite{raj17,wil08,mer13,iko19,raj19}. Since these results do not include a uniqueness part however it seems unclear how to define a suitable notion of analytic equivalence in the settings therein.

\subsection{Construction of essential metrics}

Our construction of essential metrics relies on the \emph{$p$-essential infimum} over curve families and extends the construction of the essential metric in \cite{alb18} to the case $p<\infty$; cf. \eqref{eq:ess} and Definition \ref{def:essinf}. %
For a metric measure space $X$, $p\in[1,\infty]$, and a family of curves $\Gamma$ in $X$, we define its $p$-essential length $ess \ell_p(\Gamma)\in [0,\infty]$ as the essential infimum of the length function on $\Gamma$ with respect to $p$-modulus. The \emph{$p$-essential distance} $d_{p}':X\times X \rightarrow [0,\infty]$ is defined by
\begin{equation}
d_{p}'(x,y):= \lim_{\delta \rightarrow 0} ess \ell_p \Gamma(B(x,\delta),B(y,\delta)),\quad x,y\in X,
\end{equation}
where $\Gamma(E,F)$ denotes the family of curves in $X$ joining two given measurable subsets $E,F\subset X$.
This quantity does not automatically satisfy the triangle inequality, and we consider instead the largest metric $d_p\le d_p'$; see the discussion after Definition~\ref{def:pullback}. In general, $d_p$ might take infinite values, and its finiteness is related to an abundance of curves of uniformly bounded length connecting given disjoint sets. The condition of \emph{thick quasiconvexity}, introduced in \cite{dur12}, quantifies the existence of an abundance of quasiconvex curves.

\begin{definition}\label{def:modgeod}
	Let $(X,d,\mu)$ be a metric measure space and $p \in [1,\infty]$. We say that $X$ is \textbf{$p$-thick quasiconvex} with constant $C\geq 1$ if, for all measurable subsets $E,F\subset X$ of positive measure, we have 
	\begin{equation*}
	\Mod_p\Gamma(E,F;C)>0.
	\end{equation*}
	We say that $X$ is \textbf{$p$-thick geodesic} if $X$ is $p$-thick quasiconvex with constant $C$ for every $C>1$.
\end{definition}
Here $\Gamma(E,F;C)$ denotes the family of curves $\gamma\colon [0,1]\to X$ joining $E$ and $F$ such that $\ell(\gamma)\leq C\cdot d(\gamma(0),\gamma(1))$. Note that this is equivalent to the original definition in \cite{dur12} when $X$ is infinitesimally doubling.

\begin{theorem}
	\label{thm:minmetric}
	Let $(X,d,\mu)$ be an infinitesimally doubling metric measure space which is $p$-thick quasiconvex with constant~$C$ and $1\leq p \leq \infty$. Then there exists a metric $d_p$ on $X$, for which $X_p:=(X,d_p,\mu)$ is $p$-thick geodesic and $d\leq d_p \leq Cd$. Moreover $d_p$ is minimal among metrics $\rho\ge d$ for which $X_\rho:=(X,\rho,\mu)$ is $p$-thick geodesic. 
\end{theorem}
It is known that doubling $p$-Poincar\'e spaces are $p$-thick quasiconvex for some constant depending only on the data of $X$, though the converse need not be true unless $p=\infty$; see \cite{dur12,dur12'}. For $p<q$, the assumption of $p$-thick geodecity is strictly stronger than that of $q$-thick geodecity, see Example~\ref{ex:cuspdomain} below. However, under the assumption of a suitable Poincar\'e inequality the particular value of $q$ is immaterial. Indeed, for Poincar\'e spaces the essential metrics of all indices coincide with the essential metric $d_{ess}$ introduced in \cite{alb18}. The essential metric $d_{ess}$ is given by
\begin{equation}\label{eq:ess}
d_{ess}(x,y):=ess\ell_\infty \Gamma(\{x\},\{y\}),\quad x,y \in X.
\end{equation}

\begin{proposition}\label{prop:nages}
	Let $X$ be a doubling metric measure space satisfying a $p$-Poincar\'e inequality. Then $d_q=d_{ess}$ for every $q\in [p,\infty]$.
\end{proposition}
In the proof of Theorem~\ref{thm:discmain} we will set $\widehat{Z}:=(Z,d_2)$. A posteriori, the chosen metric agrees with the essential metric~\eqref{eq:ess}. However from this characterization it is unclear that the disc obtained in Theorem~\ref{thm:discmain} is unique. This follows from the use of 2-modulus in the construction of the metric, and its quasi-invariance under quasisymmetric maps.\par

\subsection{The Sobolev-to-Lipschitz property, thick geodecity, and volume rigidity}
The Sobolev-to-Lipschitz property and thick geodecity are defined and equivalent in a much broader framework.
\begin{definition}\label{def:sobtolip}
 A metric measure space $X$ is said to have the $p$-Sobolev-to-Lipschitz property, if every $f\in \Ne{p}{X}$ for which the minimal $p$-weak upper gradient~$g_f$ satisfies $g_f\le 1$ almost everywhere has a 1-Lipschitz representative.
\end{definition}

To illustrate this equivalence, and contrast the difference of the Sobolev-to-Lipschitz property with merely being geodesic, recall that a proper metric space is geodesic if and only if, whenever a function~$f$ has (genuine) upper gradient~$1$, it is $1$-Lipschitz. The $p$-Sobolev-to-Lipschitz property requires we can reach the same conclusion from the weaker assumption that $f$ has \emph{$p$-weak} upper gradient~$1$. The next result translates this condition to having an abundance of nearly geodesic curves in the space.
\begin{theorem}\label{thm:sobtolip}
Let $X$ be a infinitesimally doubling metric measure space and $p\ge 1$. Then $X$ has the $p$-Sobolev-to-Lipschitz property if and only if it is $p$-thick geodesic.
\end{theorem}

These properties are also related to \emph{volume rigidity}, which asks whether any volume preserving $1$-Lipschitz map of a given class is an isometry. Recall that a Borel map $f:(X,\mu)\to (Y,\nu)$ between metric measure spaces is volume preserving if $f_*\mu=\nu$. When $X$ and $Y$ are $n$-rectifiable, we equip them with the Hausdorff measure $\Ha^n$, unless otherwise explicitly stated. Volume preserving 1-Lipschitz maps between rectifiable spaces are essentially length preserving, see Proposition \ref{thm:volrig}, but need not be homeomorphisms. For volume preserving 1-Lipschitz homeomorphisms with quasiconformal inverse, the Sobolev-to-Lipschitz property guarantees isometry.

\begin{proposition}\label{prop:volrig}
	Let $f:X\to Y$ be a volume preserving 1-Lipschitz homeomorphism between $n$-rectifiable metric spaces. If $f\inv$ is quasiconformal and $Y$ has the $n$-Sobolev-to-Lipschitz property, then $f$ is an isometry.
\end{proposition}

Note that the inverse of a quasiconformal map is not quasiconformal in general; cf. \cite[Remark 4.2]{wil12b}. This is guaranteed under the assumption that the spaces are Ahlfors $n$-regular and $X$ supports an $n$-Poincar\'e inequality. The assumption that $Y$ has the Sobolev-to-Lipschitz property is essential for the validity of Proposition \ref{prop:volrig}; cf. Example \ref{ex:volrigfail}. Thus, Proposition \ref{prop:volrig} connects volume rigidity, the Sobolev-to-Lipschitz property and thick geodecity under these fairly restrictive assumptions, and will be used in the proof of Theorem \ref{thm:discmain}. 

\subsection{Application: Essential minimal surfaces and essential pull-back metrics}\label{sec:app}
For the definitions and terminology in the forthcoming discussion, we refer the reader to Section \ref{sec:preli}.

A smooth map $u:M\rightarrow X$ between Riemannian manifolds gives rise to a pull-back metric on $M$. In the case that $M$ and $X$ are merely metric spaces and $u$ a continuous map, various nonequivalent definitions of pull-back metric are discussed in~\cite{pet19,pet10}. The metric
\begin{equation*}
d_u(p,q):=\inf_\gamma \ell(u \circ \gamma),\quad p,q\in M,
\end{equation*} 
where $\gamma$ varies over all continuous curves in $M$ joining $p$ to $q$, was used in \cite{lyt18,sta18,pet19} to study solutions $u:\overline\D\to X$ of the Plateau problem in~$X$, also referred to as \emph{minimal discs}. If $X$ satisfies a quadratic isoperimetric inequality and $u$ bounds a Jordan curve of finite length, then~$d_u$ gives rise to a factorization
\begin{equation*}
\begin{tikzcd}
\overline\D \arrow{rr}{u} \arrow[swap]{dr}{P_u} && X\\
& Z_u \arrow[swap]{ur}{\bar u}
\end{tikzcd}
\end{equation*}
where:
\begin{itemize}
\item[(1)] $Z_u$ is a geodesic disc which satisfies the same quadratic isoperimetric inequality as $X$.
\item[(2)] $P_u$ is a minimal disc bounding $\partial Z_u$ and a uniform limit of homeomorphisms.
\item[(3)] $\bar{u}$ is $1$-Lipschitz and $\ell(P_u\circ \gamma)=\ell(u\circ \gamma)$ for every curve $\gamma$ in $\overline{\D}$.
\end{itemize}
The fact that $Z_u$ satisfies the same quadratic isoperimetric inequality as $X$ reflects the fact that the intrinsic curvature of a classical minimal surface is bounded above by that of the ambient manifold; cf. \cite{lyt18'}.

 
Our definition of essential metrics suggests the following variant of pull-back metrics in the minimal disc setting. Namely for a family of curves $\Gamma$ in $\overline{\D}$ let $ess\ell_{u}(\Gamma)$ be the essential infimum of $\ell(u\circ \gamma)$ with respect to $2$-modulus; cf. Definition \ref{def:essinf}. Then for $p,q\in \overline{\D}$ we define the essential pull-back metric by
\begin{equation*}
\widehat{d}_u(p,q):=\lim_{\delta\rightarrow 0} ess\ell_{u}\Gamma(B_\delta(p),B_\delta(q)). 
\end{equation*}
This construction gives rise to the following factorization result.

\begin{theorem}\label{cor:plateau}
Let $X$ be a proper metric space satisfying a $(C,l_0)$-quadratic isoperimetric inequality. Let $u:\overline{\D}\rightarrow X$ be a minimal disc bounding a Jordan curve~$\Gamma$ which satisfies a chord-arc condition. Then there exists a geodesic disc $\widehat Z_u$ and a factorization 
\begin{equation*}
\begin{tikzcd}
\overline\D \arrow{rr}{u} \arrow[swap]{dr}{\widehat{P}_u} && X\\
& \widehat Z_u \arrow[swap]{ur}{\widehat u}
\end{tikzcd}
\end{equation*}
such that
\begin{itemize}
\item[(1)] $\widehat Z_u$ has the Sobolev-to-Lipschitz property and 
satisfies a $(C,l_0)$-quadratic isoperimetric inequality,
\item[(2)] $\widehat{P}_u$ a minimal disc bounding $\partial \widehat{Z}_u$ and a uniform limit of homeomorphisms, and
\item[(3)] $\widehat{u}:\widehat{Z}_u\rightarrow X$ is $1$-Lipschitz and for $2$-almost every curve $\gamma$ in $\overline{\D}$ one has $\ell(\widehat{P}_u \circ \gamma)=\ell(u \circ \gamma)$.
\end{itemize}
Furthermore the factorzation is maximal in the following sense: If $\widetilde Z$ is a metric disc and $u=\tilde u\circ\widetilde P$, where $\widetilde P:\overline\D\to \widetilde Z$ and $\tilde u:\widetilde Z\to X$ satisfy $(2)$ and $(3)$, then there exists a surjective $1$-Lipschitz map $f:\widehat{Z}_u\rightarrow \widetilde Z$ such that $\widetilde P= f\circ \widehat{P}_u$.

\end{theorem}
A Jordan curve is said to satisfy a chord-arc condition if it is bi-Lipschitz equivalent to $S^1$.
The fact that the map $\widehat{P}_u$ is a minimal disc implies additional regularity. Indeed, $\widehat P_u$ is infinitesimally quasiconformal, H\"older continuous and has global higher integrability; cf. \cite{lyt17,lyt17',lyt17'',lyt18}.\par

Compared to the construction in \cite{lyt18}, we trade off some regularity of $\overline{u}$ in exchange for the Sobolev-to-Lipschitz property on $\widehat{Z}_u$. Indeed, $\widehat u$ only preserves the length of almost every curve, while $\widehat Z_u$ is nicer in analytic and geometric terms. Note however that some geometric properties of $\widehat{Z}_u$, established for $Z_u$ in \cite{lyt18,sta18}, remain open for $\widehat Z_u$. In particular we do not know whether $\Ha^1(\partial\widehat Z_u)<\infty$.\par 
Note that, in contrast to Theorem~\ref{thm:discmain}, no characterization in terms of the Sobolev-to-Lipschitz property in Theorem~\ref{cor:plateau} is possible. In fact it may happen that $Z_u$ already has the Sobolev-to-Lipschitz property but still $\widehat{Z}_u$ and $Z_u$ are not isometric, see Example~\ref{ex:segcollapsed} below.\par
More generally when $M$ supports a $Q$-Poincar\'e inequality, where $Q\ge 1$ satisfies~\eqref{eq:relvolgrowth}, maps $u\in\Neml pMX$, with \emph{a priori higher integrability} $p>Q$, give rise to a $p$-essential pull-back metric on $M$ and the resulting metric measure space has the \emph{$Q$-Sobolev-to-Lipschitz property}; cf. Definition \ref{def:sobtolip} and Theorem~\ref{thm:pullbacksobtolip1}. Higher integrability of quasiconformal maps is known to hold when $M$ and $X$ both have $Q$-bounded geometry, and also in the situation of \cite[Theorem 1.4]{lyt17''}, but not in general; cf. \cite{hei98,hei00,kor09}. Together with the Poincar\'e inequality and Morrey's embedding higher integrability implies finiteness of the essential pull-back metric. The assumption in Theorem~\ref{cor:plateau} that $\Gamma$ is not only of finite length but satisfies a chord-arc condition is needed to guarantee the global higher integrability.

\subsection{Organization} 

The paper is organized as follows. After the preliminaries in Section \ref{sec:preli}, we treat the equivalence of thick geodecity and the Sobolev-to-Lipschitz property in Section \ref{sec:equivalence}, where we prove Theorem \ref{thm:sobtolip}. Volume rigidity is discussed in Section \ref{sec:volrig}, which includes the proof of Proposition \ref{prop:volrig}.

The construction of essential metrics is presented in Section \ref{sec:esspullback}. We develop some general tools in Section \ref{sec:regcurves}, and prove Theorem \ref{thm:minmetric} and Proposition~\ref{prop:nages} in Section~\ref{sec:sobtolip}. In Section~\ref{sec:pullbackhomeo} we apply our constuction to Sobolev maps from a Poincar\'e space, and derive some basic properties of the resulting metric measure space; cf. Theorem \ref{thm:pullbacksobtolip}.

Finally, in Section \ref{sec:discs}, we apply the results from Sections~\ref{sec:sobtolip} and~\ref{sec:pullbackhomeo} to obtain Theorems~\ref{thm:discmain} and~\ref{cor:plateau}. First in Section~\ref{sec:spheres} we prove Theorem~\ref{thm:discmain}. Then in Section~\ref{sec:discss} we prove a weaker but more general result for discs, cf. Theorem~\ref{maxdiscgen}, and finally in Section~\ref{sec:plat} we discuss Theorem~\ref{cor:plateau}.

\subsection*{Acknowledgements}
The authors would like to thank Alexander Lytchak and Stefan Wenger for helpful discussions.

\section{Preliminaries}\label{sec:preli}
We refer to the monographs \cite{HKST07,bjo11} for the material discussed below. Given a metric space $(X,d)$, balls in $X$ are denoted $B(x,r)$ and, if $B=B(x,r)\subset X$ is an open ball and $\sigma>0$, $\sigma B$ denotes the ball with the same center as $B$ and radius~$\sigma r$. The Hausdorff measure on $X$ is denoted $\Ha^n$ or, if we want to stress the space or metric, by $\Ha^n_X$ or $\Ha^n_d$. The normalizing constant is chosen so that $\Ha^n_{\R^n}$ agrees with the Lebesgue measure.

By a measure $\mu$ on a metric space $X$ we mean an outer measure which is Borel regular. A triple $(X,d,\mu)$ where $(X,d)$ is a proper metric space and $\mu$ is a measure on $X$ which is non-trivial on balls, is called a \emph{metric measure space}. Note that all metric measure spaces in our convention are complete, separable and every closed bounded subset is compact. Furthermore all balls have positive finite measure and in particular the measures are Radon, cf \cite[Corollary 3.3.47]{HKST07}. We often abbreviate $X=(X,d,\mu)$ when the metric and measure are clear from the context.

Given a curve $\gamma:[a,b]\to X$, we denote by  $\ell(\gamma)$ its length. When $\gamma$ is absolutely continuous we write $|\gamma_t'|$ for the metric speed (at $t\in [a,b]$) of $\gamma$. Moreover, when $\gamma$ is rectifiable (i.e. $\ell(\gamma)<\infty$) we denote by $\bar\gamma:[a,b]\to X$ the \emph{constant speed parametrization} of $\gamma$. This is the (Lipschitz) curve satisfying
\[
\gamma(a+(b-a)\ell(\gamma|_{[a,t]})/\ell(\gamma))=\bar\gamma(t),\quad t\in [a,b].
\]
The \emph{arc length parametrization} $\gamma_s$ of $\gamma$ is the reparametrization of $\bar\gamma$ to the interval $[0,\ell(\gamma)]$. With the exception of Proposition \ref{prop:upreg}, we assume curves are defined on the interval $[0,1]$.

Given a metric measure space $X$, a Banach space $V$, and $p\ge 1$, we denote by $L^p(X;V)$ and $L^p_{\loc}(X;V)$ the a.e.-equivalence classes of $p$-integrable and locally $p$-integrable $\mu$-measurable maps $u\colon X\to V$. We also denote $L^p(X):=L^p(X;\R)$ and $L^p_{\loc}(X):=L^p_{\loc}(X;\R)$ and abuse notation by writing $f\in L^p(X;V)$ or $f\in L_{\loc}^p(X;V)$ for \emph{maps} $f$ (instead of equivalence classes).

\subsubsection*{Properties of measures} A measure $\mu$ on $X$ is called \emph{doubling} if there exists $C\ge 1$ such that
\begin{equation}
\mu(B(x,2r))\le C\mu(B(x,r))
\end{equation}
for every $x\in X$ and $0<r<\diam X$. The least such constant is denoted $C_\mu$ and called the doubling constant of $\mu$. Doubling measures satisfy a relative volume lower bound
\begin{equation}\label{eq:relvolgrowth}
C\left(\frac{\diam B}{\diam B'}\right)^Q\le \frac{\mu(B)}{\mu(B')}\textrm{ for all balls }B\subset B'\subset X,
\end{equation}
for some constants $C>0$ and $Q\le \log_2C_\mu$ depending only on $C_\mu$. The opposite inequality with the same exponent $Q$ need not hold. If there are constants $C,Q >0$ such that
\begin{align*}
\frac 1C r^Q\le \mu(B(x,r))\le Cr^Q,\quad 0<r<\diam X,
\end{align*}
we say that $\mu$ is \emph{Ahlfors $Q$-regular}.

If $\mu$ is doubling and
\[
\M_rf(x)=\sup_{0<s<r}\dashint_{B(x,s)}|f|\ud\mu
\]
denotes the (restricted) centered maximal function of $f$ at $x\in X$ the sublinear operator $\M_r$ satisfies the usual boundedness estimates
\begin{align}\label{eq:maximalfunctionbound}
\|\M_rf\|_{L^p(X)}\le C\|f\|_{L^p(X)}\ (p>1),\quad \mu(\{ \M_rf>\lambda \})\le C\frac{\|f\|_{L^1(X)}}{\lambda},\ \lambda>0,
\end{align}
for some constant $C$ depending only on $C_\mu$. Consequently almost every point of a $\mu$-measurable set $E\subset X$ is a Lebesgue density point.

If the measure $\mu$ satisfies the \emph{infinitesimal doubling condition} 
\begin{equation*}
\limsup_{r\to 0}\frac{\mu(B(x,2r))}{\mu(B(x,r))}<\infty
\end{equation*}
for $\mu$-almost every $x\in X$, the claim about density points still remains true.

\begin{proposition}\cite[Theorem 3.4.3]{HKST07}\label{prop:hkst}
	If $(X,d,\mu)$ is infinitesimally doubling metric measure space and $f\in L^1_{\loc}(X)$ then 
	\[
	f(x)=\lim_{r\to 0}\dashint_{B(x,r)}f\ud\mu
	\]
	for $\mu$-almost every $x\in X$. In particular, $\mu$-almost every point of a Borel set $E\subset X$ is a Lebesgue density point.
\end{proposition}

\subsubsection*{Sobolev spaces}
Let $(X,d,\mu)$ be a metric measure space, $Y$ a metric space, and $p\ge 1$. If $u:X\to Y$ is a map and $g:X\to [0,\infty]$ is Borel, we say that $g$ is an \emph{upper gradient} of $u$ if
\begin{equation}\label{eq:ug}
d_Y(u(\gamma(b),\gamma(a))\le \int_\gamma g
\end{equation}
for every curve $\gamma\colon [a,b]\to X$. Recall that the \emph{line integral} of $g$  over $\gamma$ is defined as
\[
\int_\gamma g:=\int_0^{\ell(\gamma)}g(\gamma_s(t))\ud t
\]
if $\gamma$ is rectifiable and $\infty$ otherwise. Suppose $Y=V$ is a separable Banach space. A $\mu$-measurable map $u\in L^p(X;V)$ is called $p$-Newtonian if it has an upper gradient $g\in L^p(X)$. The Newtonian seminorm is
\begin{align}\label{eq:sobnorm}
\|u\|_{1,p}=\left(\|u\|_{L^p(X;V)}^p+\underset{g}{\inf}\|g\|_{L^p(X)}^p\right)^{1/p},
\end{align}
where the infimum is taken over all upper gradients $g$ of $u$. The Newtonian space $\Nem pXV$  is the vector space of equivalence classes of $p$-Newtonian maps, where two maps $u,v:X\to V$ are equivalent if $\|u-v\|_{1,p}=0$. The quantity (\ref{eq:sobnorm}) defines a norm on $\Nem pXV$ and $(\Nem pXV,\|\cdot\|_{1,p})$ is a Banach space.

We say that a $\mu$-measurable map $u:X\to V$ is locally $p$-Newtonian if every point $x\in X$ has a neighbourhood $U$ such that
\[
u|_U\in \Nem pUV,
\]
and denote the vector space of locally $p$-Newtonian maps by $\Neml pXV$.

\subsubsection*{Minimal upper gradients} Let $\Gamma$ be a family of curves in $X$. We define the $p$-modulus of $\Gamma$ as
\[
\Mod_p(\Gamma):=\inf\left\{ \int_X\rho^p\ud\mu:\ \rho\colon X\to [0,\infty] \textrm{ Borel, } \int_\gamma\rho\ge 1\textrm{ for every }\gamma\in \Gamma \right\}.
\]
We say that a Borel function $g\colon X\to [0,\infty]$ is a $p$-weak upper gradient of $u$ if there is a path family $\Gamma_0$ with $\Mod_p(\Gamma_0)=0$ so that (\ref{eq:ug}) holds for all curves $\gamma\notin\Gamma_0$. The infimum in \eqref{eq:sobnorm} need not be attained by upper gradients of $u$ but there is a minimal $p$-weak upper gradient $g_u$ of $u$ so that
\[
\|g_u\|_{L^p(X)}=\inf_g\|g\|_{L^p(X)}.
\]
The minimal $p$-weak upper gradient is unique up to sets of measure zero.

If $Y$ is a complete separable metric space then there is an isometric embedding $\iota:Y\to V$ into a separable Banach space. We may define 
\[
\Neml pXY=\{ u\in \Neml pXV:\ \iota(u(x))\in \iota(Y)\ \textrm{ for $\mu$-almost every }x\in X \}
\]
We remark that maps $u\in \Neml pXY$ have minimal $p$-weak upper gradients that are unique up to equality almost everywhere and do not depend on the embedding.

\subsubsection*{Poincare inequalities} A metric measure space $X=(X,d,\mu)$ supports a \emph{$p$-Poincar\'e inequality} if there are constants $C,\sigma\ge 1$ so that
\begin{equation}
\dashint_B|u-u_B|\ud\mu\le C\diam B\left(\dashint_{\sigma B}g^p\ud\mu\right)^{1/p}
\end{equation}
whenever $u\in L^1_{\loc}(X)$, $g$ is an upper gradient of $u$, and $B$ is a ball in $X$. Doubling metric measure spaces supporting a Poincar\'e inequality enjoy a rich theory.

\begin{theorem}[Morrey embedding]\label{thm:morrey}
	Let $(X,d,\mu)$ be a complete doubling metric measure space, where the measure satisfies (\ref{eq:relvolgrowth}) for $Q\ge 1$, and suppose $X$ supports a $Q$-Poincar\'e inequality. If $p>Q$, there is a constant $C\ge 1$ depending only on the data of $X$ so that for any $u\in \Neml pXY$ and any ball $B\subset X$ we have
	\begin{equation*}
	d_Y(u(x),u(y))\le C(\diam B)^{Q/p}d(x,y)^{1-Q/p}\left(\dashint_{\sigma B}g_u^p\ud\mu\right)^{1/p},\quad x,y\in B.
	\end{equation*}
\end{theorem}
Here the \emph{data} of $X$ refers to $p$, the doubling constant of the measure and the constants in the Poincar\'e inequality.
\subsubsection*{Metric differentials}

Let $E\subset \R^n$ be measurable and $f:E\to X$ be a Lipschitz map into a metric space $X$. By a fundamental result of Kirchheim \cite{kir94}, $f$ admits a \emph{metric differential} $\md_xf$ at almost every point $x\in E$. The metric differential is a seminorm on $\R^n$ and such that
\[
|\md_xf(y-z)-d(f(y),f(z))|=o(d(x,y)+d(x,z))\textrm{ whenever }y,z\in E.
\]
If $\Omega\subset\R^n$ is a domain and $p\ge 1$, a Sobolev map $u\in\Neml p{\Omega}X$ admits an \emph{approximate metric differential} $\apmd_x u$ for almost every $x\in\Omega$; cf. \cite{kar07} and \cite{lyt17''}. If $\Omega\subset M$ is a domain in a Riemannian $n$-manifold, the approximate metric differential of $u\in \Nem p\Omega X$ may be defined, for almost every $x\in \Omega$, as a seminorm on $T_xM$ by 
\begin{equation}\label{eq:apmd}
\apmd_xu:=\apmd_{\psi(x)}(u\circ\psi\inv)\circ\ud\psi(x)
\end{equation}
for any chart $\psi$ around $x$; see the discussion after \cite[Definition 2.1]{fit19} and the references therein for the details.

We formulate the next results for maps defined on domains of manifolds; that is, we assume that $\Omega\subset M$ for a Riemannian $n$-manifold $M$. Note that, although in the references the statements are proved in the Euclidean setting, the definition \eqref{eq:apmd} easily implies the corresponding statements for general Riemannian manifolds.
For the following area formula for Sobolev maps with Lusin's property ($N$), recall that a Borel map $f:(X,\mu)\to (Y,\nu)$ is said to have Lusin's property ($N$) if $\nu(f(E))=0$ whenever $\mu(E)=0$.

\begin{theorem}\cite[Theorems 2.4 and 3.2]{kar07}\label{thm:rectarea}
Let $u\in \Neml p{\Omega}X$ for some $p>n$. Then $u$ has Lusin's property ($N$), and
\begin{equation}
\int_\Omega\varphi(x)J(\apmd_xu)\ud x=\int_X\left(\sum_{x\in u\inv(y)}\varphi(x)\right)\ud\Ha^n(y)
\end{equation}
for any Borel function $\varphi:\Omega\to [0,\infty]$.
\end{theorem}
Here the \emph{Jacobian} of a seminorm $s$ on $\R^n$ is defined by
\begin{equation}\label{eq:seminormjac}
J(s)=\alpha(n)\left(\int_{S^{n-1}}s(v)^{-n}\ud \Ha^{n-1}(v)\right)\inv
\end{equation}
Following \cite{lyt17''} we also define the \emph{maximal stretch} of a seminorm $s$ by
\[
I_+(s)=\max\{ s(v)^n:\ |v|=1 \}.
\]
Note that the inequality $J(s)\le I_+^n(s)$ always holds, and $I_+(\apmd_xu)=g_u$ almost everywhere for $u\in \Neml n{\Omega}X$, see \cite[Section 4]{lyt17''}. We say that $s$ is \emph{$K$-quasiconformal} if $I_+^n(s)\le KJ(s)$.  A Sobolev map $u\in \Neml n\Omega X$ is called \emph{infinitesimally $K$-quasiconformal}, if $\apmd_x u$ is $K$-quasiconformal for almost every $x\in\Omega$.

The \emph{(Reshetnyak) energy} and \emph{(parametrized Hausdorff) area} of $u\in \Neml n{\Omega}X$ are defined as
\begin{align*}
E_+^n(u)&:=\int_\Omega I_+^n(\apmd_xu)\ud x\\
Area(u)&:=\int_\Omega J(\apmd_xu)\ud x.
\end{align*}

\subsubsection*{Quasiconformality}
We refer to \cite{hei95,hei96,hei98,hei01} for various definitions of quasiconformality and their relationship in metric spaces, and only mention what is sometimes known as \emph{geometric quasiconformality}. A homeomorphism $u:(X,\mu)\to (Y,\nu)$ between metric measure spaces is said to be \emph{$K$-quasiconformal} with "index" $Q\ge 1$, if
\begin{equation}\label{eq:qc}
\Mod_Q\Gamma\le K\Mod_Qu(\Gamma)
\end{equation}
for any curve family $\Gamma$ in $X$. Without further assumptions on the geometry of the spaces, the modulus condition \eqref{eq:qc} is fundamentally one-sided; cf. the discussion in the introduction. We refer to \cite{wil12b} for an equivalent characterization in terms of \emph{analytic quasiconformality} and note that there is a corresponding result for \emph{monotone maps}. Recall that a map is monotone if the preimage of every point is a connected set. 
\begin{proposition}\cite[Proposition 3.4]{lyt17}
	Let $M$ be $\overline\D$ or $S^2$. If $u\in \Nem 2MX$ is a continuous, surjective, monotone and infinitesimally $K$-quasiconformal map into a complete metric space $X$, then
	\[
	\Mod_2\Gamma\le K\Mod_2u(\Gamma)
	\]
	for any curve family $\Gamma$ in $\overline\D$.
\end{proposition}
We remark that in \cite{lyt17} the result is proven for $M=\overline\D$. The same proof however carries over to $S^2$ with minor modifications.

\subsubsection*{Quadratic isoperimetric inequality and minimal discs} Given a proper metric space~$X$ and a closed Jordan curve $\Gamma$ in $X$, we denote
\[
\Lambda(\Gamma,X):=\{u\in \Nem 2{\overline{\D}}X:\ u \textrm{ spans }\Gamma  \}.
\]
Here, a function $u\in \Nem 2{\overline\D}X$ is said to span $\Gamma$ if $u|_{S^1}$ agrees a.e. with a monotone parametrization of $\Gamma$. We call $u\in \Lambda(\Gamma,X)$ a \emph{minimal disc} spanning~$\Gamma$ if~$u$ has least area among all maps in~$\Lambda(\Gamma,X)$ and is furthermore of minimal Reshetnyak energy~$E^2_+$ among all area minimizers in $\Lambda(\Gamma,X)$. In~\cite{lyt17''} it has been shown that, as soon as $\Lambda(\Gamma,X)\neq \varnothing$, there exists a minimal disc~$u$ spanning $\Gamma$ and every such~$u$ is infinitesimally $\sqrt{2}$-quasiconformal. 

\begin{definition}
	A metric space $X$ is said to satisfy a $(C,l_0)$-quadratic isoperimetric inequality if, for any Lipschitz curve $\gamma:S^1\rightarrow X$ with $\ell(\gamma)<l_0$, there exists $u\in \Nem 2{\overline{\D}}X$ with $u_{|S^1}=\gamma$ and 
	\[
	Area(u)\le C\cdot \ell(\gamma)^2.
	\]
\end{definition}
If $X$ satisfies a quadratic isoperimetric inequality then minimal discs $u:\overline{\D}\rightarrow X$ lie in $N^{1,p}_{loc}(\D,X)$ for some~$p>2$ and, if furthermore~$\Gamma$ satisfies a chord-arc condition, then $u\in N^{1,p}(\overline{\D},X)$.\par 
If a proper geodesic metric space $Z$ is homeomorphic to a $2$-manifold, then it satisfies a $(C,l_0)$-quadratic isoperimetric inequality if and only if every Jordan curve~$\Gamma$ in~$Z$ such that $\ell(\Gamma)<l_0$ bounds a Jordan domain $U\subset Z$ which satisfies
\begin{equation}
\label{eq:qii}
\mathcal{H}^2(U)\leq C\cdot \ell(\Gamma)^2.
\end{equation}
This follows from the proof of~\cite[Theorem 1.4]{lyt17} together with the observation leading to~\cite[Corollary 1.5]{cre19'}. We will use this fact in Section~$6$ to establish the quadratic isoperimetric inequality for $\widehat Z_u$ in Theorem~\ref{cor:plateau}.

\section{A geometric characterization of the Sobolev-to-Lipschitz property}\label{sec:equivalence}

In this section we prove Theorem \ref{thm:sobtolip}.
We assume throughout this section that $X$ is a metric measure space and $p\ge 1$. We will need the following consequence of the Sobolev-to-Lipschitz property for measurable maps with 1 as $p$-weak upper gradients. 
\begin{lemma}\label{lem:truesobtolip}
	Suppose $X$ has the $p$-Sobolev-to-Lipschitz property, and let $V$ be a separable Banach space. Then any measurable function $f:X\to V$ with 1 as $p$-weak upper gradient has a 1-Lipschitz representative.
\end{lemma}
\begin{proof}
	We prove the claim first when $V=\R$. Fix $x_0\in X$ and consider the function
	\[
	w_k:=(k-\dist(\cdot,B(x_0,k))_+
	\]
	for each $k\in \N$. Note that $w_k$ is 1-Lipschitz and $w_k\in\Ne pX$. The functions
	\[
	f_k:=\min\{ w_k,f \}
	\]
	have 1 as $p$-weak upper gradient (see \cite[Proposition 7.1.8]{HKST07}) and $f_k\in\Ne pX$. By the Sobolev-to-Lipschitz property there is a set $N\subset X$ with $\mu(N)=0$ and 1-Lipschitz functions $\bar f_k$ such that of $f_k(x)=\bar f_k(x)$ if $x\notin N$, for each $k\in\N$. Since $f_k\to f$ pointwise everywhere we have, for $x,y\notin N$, that 
	\[
	|f(x)-f(y)|=\lim_{k\to\infty}|\bar f_k(x)-\bar f_k(y)|\le d(x,y).
	\]
	Thus $f$ has a 1-Lipschitz representative.

	Next, let $V$  be a separable Banach space and let $\{x_n\}\subset V$ be a countable dense set. Given $f$ as in the claim, the functions $f_n:=\dist(x_n,f)$ have 1 as $p$-weak upper gradient, and thus there is a null set $E\subset X$ and 1-Lipschitz functions $\bar f_n$ so that $f_n(x)=\bar f_n(x)$ whenever $n\in\N$ and $x\notin E$. For $x,y\notin N$, we have 
	\[
	\|f(x)-f(y)\|_V=\sup_n|f_n(x)-f_n(y)|=\sup_n|\bar f_n(x)-\bar f_n(y)|\le d(x,y).
	\]
	We again conclude that $f$ has a 1-Lipschitz representative.
\end{proof}

\begin{proposition}\label{prop:sobtoClip}
	Let $X$ be infinitesimally doubling and $p$-thick quasiconvex with constant $C$. Then every $u\in N^{1,p}(X)$ satisfying $g_u\leq 1$, has a $C$-Lipschitz representative.
\end{proposition}
\begin{proof}
	Let $\Gamma_0$ a path family of zero $p$-modulus such that
	\[
	|u(\gamma(1))-u(\gamma(0))|\le \int_\gamma 1=\ell(\gamma)
	\]
	whenever $\gamma\notin \Gamma_0$.
	By Lusin's theorem there is a decreasing sequence of open sets $U_m\subset X$ such that, for each $m\in\N$, $\mu(U_m)<2^{-m}$ and $u|_{X\setminus U_m}$ is continuous. Fix $m\in \N$, and let $x$ and $y$ be distinct density points of $X\setminus U_m$. Then for every $\delta>0$ we have that
	\[
	\Mod_p(\Gamma(B(x,\delta)\setminus U_m,B(y,\delta)\setminus U_m;C))>0.
	\]
	For $\gamma^\delta\in \Gamma(B(x,\delta)\setminus U_m,B(y,\delta)\setminus U_m;C)\setminus \Gamma_0$ one has
	\[
	|u(\gamma^\delta(1))-u(\gamma^\delta(0))|\leq \ell(\gamma)\leq C\cdot d(\gamma^\delta(0),\gamma^\delta(1))\leq  C d(x,y)+2C\delta.
	\]
	Letting $\delta\rightarrow 0$ and using the continuity of $u|_{X\setminus U_m}$ we obtain
	\begin{equation}
	\label{eqlip}
	|u(x)-u(y)|\leq C d(x,y).
	\end{equation}
	Letting $m \rightarrow \infty$ one sees that \eqref{eqlip} holds for $\mu$-almost every $x,y \in X$ and hence $u$ has a $C$-Lipschitz representative.
\end{proof}

\begin{proof}[Proof of Theorem \ref{thm:sobtolip}]
One implication is directly implied by the preceeding proposition. Indeed, assume $X$ is $p$-thick geodesic and let $f\in \Ne pX$ have $1$ as a $p$-weak upper gradient. By Proposition~\ref{prop:sobtoClip},  $f$ has a $C$-Lipschitz representative for every $C>1$. So the continuous representative of $f$ is $1$-Lipschitz.\par 
Now assume $X$ has the $p$-Sobolev-to-Lipschitz property but is not $p$-thick geodesic. Let $E,F\subset X$ measurable and $C>1$ be such that $\mu(E)\mu(F)>0$ and $\Mod_p\Gamma(E,F;C)=0$. By looking at density points of $E$ and $F$ respectively we may assume without loss of generality that
\[
0<\dist(E,F)=:D\ \ \textrm{ and } \ \ \textnormal{diam}(E),\textnormal{diam}(F)\leq \frac{(C-1)D}{4C}.
\]
There exists a non-negative Borel function $g\in L^p(X)$ for which $\displaystyle \int_\gamma g=\infty$ for every $\gamma\in \Gamma(E,F;C)$. Denote
\[
\Gamma_0:=\left\{ \gamma: \int_\gamma g=\infty \right\},
\]
whence $\Gamma(E,F;C)\subset \Gamma_0$, $\Mod_p(\Gamma_0)=0$ and for $\gamma_1,\gamma_2\notin \Gamma_0$ one has $\gamma_1\cdot \gamma_2 \notin \Gamma_0$.
Define the function $v:X\rightarrow [0,\infty]$ by
\[
v(x)=\lim_{n\to\infty}\inf\left\{ \int_\gamma(1+g/n):\ \gamma\in\Gamma(E,x) \right\}.
\]
The function $v$ is measurable by \cite[Theorem 9.3.1]{HKST07} and satisfies $v|_{E}\equiv 0$. We also have 
\begin{equation}\label{contra}
|v(y)-v(x)|\ge C \cdot D\ge C\cdot\left(d(x,y)-\frac{(C-1)D}{2C}\right)\ge \frac{C+1}{2} d(x,y)
\end{equation}
whenever $y\in F$ and $x\in E$.\par 
For any $\gamma\notin \Gamma_0$  by closedness under composition one has $v(\gamma(0))=\infty$ iff $v(\gamma(1))=\infty$. Furthermore if $v(\gamma(0))\neq \infty$, then
\begin{equation}
|v(\gamma(1))-v(\gamma(0))|\leq \ell(\gamma)=\int_\gamma 1
\end{equation}
So $g \equiv 1$ is a $p$-weak upper gradient of $v$. 

If we had that $v \in N^{1,p}(X)$, the Sobolev-to-Lipschitz property and \eqref{contra} would lead to a contradiction. A simple cut off argument remains to complete the proof.\\
Define $w:X\rightarrow [0,\infty]$ by \[w(x):=\left((C+2)D-d(E,x)\right)^+.\] Then $w$ has $1$ as a upper gradient, $w\in N^{1,p}(X)$ is compactly supported and for $x\in E$, $y \in F$ one has
\begin{equation}
w(y)=(C+2)D-d(E,y)\geq CD\ge \frac{C+1}{2} d(x,y).
\end{equation} 
Set $u:=\min\{v,w\}$. Then $u\in N^{1,p}(X)$ and $g_u\leq 1$. By the Sobolev-to-Lipschitz property $u$ has a $1$-Lipschitz representative. But for every $x\in E$, $y\in F$ \[ |u(x)-u(y)|=u(y)\ge \frac{C+1}{2} d(x,y).\]As $\mu(E)\mu(F)>0$ these two observations lead to a contradiction.
\end{proof}
\begin{example}
\label{ex:cuspdomain}
For $p\in [1,\infty)$ let $Z_p:=\{(x,y)\in \R^2:\ |y|\leq |x|^p\}$ be endowed with the intrinsic length metric and Lebesgue measure. Then $Z$ is $q$-thick geodesic if and only if $q>p+1$. This follows by Example~$1$ in \cite{dur12}. Since $Z_p$ is doubling, it also has the $q$-Sobolev-to-Lipschitz property if and only if $q>p+1$.
\end{example}

\section{Volume rigidity}\label{sec:volrig}

We prove Proposition \ref{prop:volrig} using the fact that volume preserving 1-Lipschitz maps between rectifiable spaces are essentially length preserving. Recall that we say a map $f:X\to Y$ between $n$-rectifiable spaces is volume preserving, if $f_*\Ha^n_X=\Ha^n_Y$.

\begin{proposition}\label{thm:volrig}
	Let $f:X\to Y$ volume preserving $1$-Lipschitz map between $n$-rectifiable metric spaces.  There exists a Borel set $N\subset X$ with $\Ha^n(N)=0$ so that for every absolutely continuous curve $\gamma$ in $X$ with $|\gamma\inv N|=0$ we have
	\[
	\ell(\gamma)=\ell(f\circ\gamma).
	\]
\end{proposition}
\begin{proof}
	Since $X$ is $n$-rectifiable, there are bounded Borel sets $E_i\subset \R^n$, and bi-Lipschitz maps $g_i:E_i\to  X$ so that the sets $g_i(E_i)\subset X$ are pairwise disjoint, and $\mathcal{H}^n(E)=0$, where
	\[
	E:=X\setminus \bigcup_i g_i(E_i).
	\]
	For each $i$, fix Lipschitz extensions $\bar g_i:\R^n\to l^\infty$ and $\bar f_i:\R^n\to l^\infty$ of $g_i$ and $f\circ g_i$, respectively (here we embed $X$ and $Y$ isometrically into $l^\infty$).
	
	\begin{lemma}\label{lem:equal}
		Let $i\in\N$. For almost every $x\in E_i$, we have
		\begin{equation}\label{eq:equal}
		\md_x \bar g_i =\md_x \bar f_i.
		\end{equation}
	\end{lemma}
	\begin{proof}[Proof of Lemma \ref{lem:equal}]
		Since $f$ is 1-Lipschitz we have, at almost every point $x\in E_i$ where both $\md_x \bar g_i$ and $\md_x\bar f_i$ exist, the inequality $\md_x \bar f_i\le \md_x \bar g_i$, which in particular implies $J(\md_x\bar f_i)\le J(\md_x\bar g_i)$ for $\leb^n$-almost every $x\in E_i$. By the area formula \cite[Corollary 8]{kir94} and the fact that $f$ is volume preserving, it follows that
		\begin{align*}
		&\int_{E_i}J(\md_x\bar g_i)\ud\leb^n(x)=\Ha^n(g_i(E_i))=\Ha^n(f(g_i(E_i)))\\
		=&\int_{E_i}J(\md_x\bar f_i)\ud\leb^n(x);
		\end{align*}
		cf. Theorem \ref{thm:rectarea}. Consequently $J(\md_x\bar g_i)= J(\md_x\bar f_i)$ for $\leb^n$-almost every $x\in E_i$. Thus $\md_x\bar g_i=\md_x\bar f_i$ for $\leb^n$-almost every $x\in E_i$.
	\end{proof}
	
	For each $i\in\N$, let $U_i$ denote the set of those density points $x$ of $E_i$ for which $\md_x \bar g_i=\md_x\bar f_i$. Lemma \ref{lem:equal} implies that $\leb^n(E_i\setminus U_i)=0$. Set 
	\[
	N:=E\cup\bigcup_ig_i(E_i\setminus U_i),
	\]
	whence $\mathcal{H}^n(N)=0$.
	
	Let $\gamma:[0,1]\to X$ be a Lipschitz path.  For each $i\in\N$, denote $K_i=\gamma\inv(g_i(U_i))\subset [0,1]$, and suppose $\gamma_i:[0,1]\to \R^n$ is a Lipschitz extension of the "curve fragment" $g_i\inv\circ\gamma:K_i\to\R^n$.

	Let $t\in K_i$ be a density point of $K_i$, for which $|\dot\gamma_t|$, $|(f\circ\gamma)_t'|$ and $\gamma_i'(t)$ exist, and $\gamma(t)\notin N$. Then $\gamma_i(t)\in U_i$ and we have
	\begin{align*}
	\md_{\gamma_i(t)}\bar g_i(\gamma_i'(t))=&\lim_{h\to 0}\frac{\|\bar g_i(\gamma_i(t)+h\gamma_i'(t))-\bar g_i(\gamma_i(t))\|_{l^\infty}}{h}\\
	=&\lim_{h\to 0}\frac{\|\bar g_i(\gamma_i(t+h))-\bar g_i(\gamma_i(t))\|_{l^\infty}+o(h)}{h}\\
	=&\lim_{\stackrel{h\to 0}{t+h\in K_i}}\frac{\|\bar g_i(\gamma_i(t+h)-\bar g_i(\gamma_i(t))\|_{l^\infty}}{h}\\
	=&\lim_{\stackrel{h\to 0}{t+h\in K_i}}\frac{d(\gamma(t+h),\gamma(t))}{h}=|\dot\gamma_t|.
	\end{align*}
	The same argument with $\bar f_i$ in place of $\bar g_i$ yields
	\[
	\md_{\gamma_i(t)}\bar f_i(\gamma_i'(t))=|(f\circ\gamma)'_t|.
	\]
	Since $\gamma_i(t)\in U_i$, we have
	\begin{align*}
	|\dot\gamma_t|=\md_{\gamma_i(t)}\bar g_i(\gamma_i'(t))=\md_{\gamma_i(t)}\bar f_i(\gamma_i'(t))=|(f\circ\gamma)'_t|.
	\end{align*}
	
	\medskip\noindent Suppose now that $\gamma:[0,1]\to X$ is a absolutely continuous path with $|\gamma\inv(N)|=0$. We may assume that $\gamma$ is constant speed parametrized. Since $|\gamma\inv (N)|=0$ it follows that the union of the sets $K_i$ over $i\in \N$ has full measure in $[0,1]$. Since, for each $i\in\N$, a.e. $K_i$ satisfies the conditions listed above, we may compute
	\begin{align*}
	\ell(\gamma)=\int_0^1|\dot\gamma_t|\ud t = \sum_i\int_{K_i}|\dot\gamma_t|\ud t=\sum_i\int_{K_i}|(f\circ\gamma)'_t|\ud t=\ell(f\circ\gamma).
	\end{align*}
\end{proof}
The proof yields the following corollary.
\begin{corollary}\label{cor:volrig}
	Let $f:X\to Y$ be a surjective 1-Lipschitz map between rectifiable spaces with $\Ha^n(X)=\Ha^n(Y)<\infty$. Then $f$ is volume preserving and, in particular, $\ell(f\circ \gamma)=\ell(\gamma)$ for $\infty$-almost every $\gamma$ in $X$.
\end{corollary}
We close this section with the proof of Proposition \ref{prop:volrig}, which connects volume rigidity and the Sobolev-to-Lipschitz property.
\begin{proof}[Proof of Proposition \ref{prop:volrig}]
	Since $f\inv$ is quasiconformal, $f\inv\circ\gamma$ is rectifiable for $n$-almost every curve $\gamma$ in $Y$. Proposition \ref{thm:volrig} implies that
	\[
	g_{f\inv}\le 1.
	\]
	Since $Y$ has the $n$-Sobolev-to-Lipschitz property, it follows that $f\inv$ has a 1-Lipschitz representative, cf. Lemma \ref{lem:truesobtolip}. By continuity of $f\inv$ it coincides with this representative.
\end{proof}

The Sobolev-to-Lipschitz property is crucial for the conclusion of Proposition~\ref{prop:volrig}, as the next example shows.
\begin{example}\label{ex:volrigfail}
	Let $Y=(\overline\D,d_w)$, where
	\[ 
	w(x)=1-(1-|x|)\chi_{[0,1]\times \{ 0\}},\quad x\in\overline\D,
	\]
	and the metric $d_w$ is given by, 
	\[
	d_w(x,y)=\underset{\gamma\in\Gamma(x,y)}{\inf} \int_\gamma w,\quad x,y\in\overline\D.
	\]
	The map $\id:\overline\D\to Y$ is a volume preserving 1-Lipschitz homeomorphism and the inverse is quasiconformal, but not Lipschitz continuous.
\end{example}
Indeed, since $[0,1]\times\{0\}$ has zero measure, we have $J(\apmd \id)=1$ almost everywhere, whence the area formula implies that $\Ha^2(\overline\D)=\Ha^2_{d_w}(\overline\D)$. 

\section{Construction of essential pull-back metrics}\label{sec:esspullback}

Let $X$ be a metric measure space and let $\overline\Gamma(X)$ denote the set of Lipschitz curves $[0,1]\to X$. Recall that $p$-almost every curve in $X$ admits a Lipschitz reparametrization. In this section we construct essential pull-back distances by Sobolev maps. The key notion here is the essential infimum of a functional over a path family.

\begin{definition}\label{def:essinf}
Let $p\ge 1$, $F:\overline\Gamma(X)\to [-\infty,\infty]$ be a function and $\Gamma\subset \overline\Gamma(X)$ a path family. Define the \emph{$p$-essential infimum} of $F$ over $\Gamma$ by 
\[
\underset{\Gamma}{\essinf_{p}}F=\underset{\gamma\in \Gamma}{\essinf_{p}}F(\gamma):=\sup_{\Mod_{p}(\Gamma_0)=0}\inf\{ F(\gamma):\ \gamma\in \Gamma\setminus\Gamma_0 \}
\]
with the usual convention $\inf\varnothing=\infty$.
\end{definition}
It is clear that the supremum may be taken over curve families $\Gamma_\rho$ for non-negative Borel functions $\rho\in L^p(X)$, where
\[
\Gamma_\rho:=\left\{ \gamma: \int_\gamma\rho=\infty \right\}.
\]

\begin{remark}\label{rmk:essinf}
We have the following alternative expression for $\essinf_p$:
\begin{align*}
\underset{\Gamma}{\essinf_{p}}F&=\max\{ \lambda>0:\ \Mod_p(\Gamma\cap \Gamma_F(\lambda))=0 \}\\
&=\inf\{ \lambda>0:\ \Mod_p(\Gamma\cap \Gamma_F(\lambda))>0 \}.
\end{align*}
Here
\[
\Gamma_F(\lambda):=\{\gamma\in\overline\Gamma(X): F(\gamma)<\lambda \}.
\]
\end{remark}

Let $Y$ be a complete metric space, $p\ge 1$, and $u\in \Neml pXY$.
Given a path family $\Gamma$ in $X$ we set
\[
ess\ell_{u,p}(\Gamma):=\underset{\gamma\in\Gamma}{\essinf_p}\ell(u\circ\gamma).
\]

\begin{definition}\label{def:pullback}
	Let $u\in \Neml pXY$. Define the \emph{essential pull-back distance}\\ $\displaystyle d'_{u,p}:X\times X\to[0,\infty]$ by
	\[
	d'_{u,p}(x,y):=\lim_{\delta\to 0}ess\ell_{u,p}(\Gamma(\bar B(x,\delta),\bar B(y,\delta)), \quad x,y\in X.
	\]
\end{definition}
In general, the essential pull-back distance may assume both values 0 and $\infty$ for distinct points, and it need not satisfy the triangle inequality. We denote by $d_{u,p}$ the maximal pseudometric not greater than $d_{u,p}'$, given by
\[
d_{u,p}(x,y)=\inf\left\{ \sum_{i=1}^nd_{u,p}'(x_i,x_{i-1}):\ x_0,\ldots,x_n\in X,\ x_0=x,\ x_n=y \right\},\quad x,y\in X.
\]
The maximal pseudometric below $d_{u,p}'$ may also fail to be a finite valued. We give two situations that guarantee finiteness and nice properties of the arising metric space. Firstly, in Section \ref{sec:sobtolip} we consider the simple case $u=\id:X\to X$ and use it to prove Theorem \ref{thm:minmetric}. Secondly, in Proposition \ref{prop:Xu} we prove that, when $X$ supports a Poincar\'e inequality and $u\in \Neml pXY$ for large enough $p$, the distance in Definition \ref{def:pullback} is a finite valued pseudometric.

For both we need the notion of regular curves, whose properties we study next.

\subsection{Regular curves}\label{sec:regcurves}
Throughout this subsection $(X,d,\mu)$ is a metric measure space. We define a metric on $\overline \Gamma(X)$ by setting 
\[
d_\infty(\alpha,\beta):=\sup_{0\le t\le 1}d(\alpha(t),\beta(t))
\]
for any two Lipschitz curves $\alpha,\beta$. By a simple application of the Arzela-Ascoli theorem it follows that $(\overline\Gamma(X),d_\infty)$ is separable.

\begin{definition}
A curve $\gamma\in\overline \Gamma (X)$ is called $(u,p)$-regular if $u\circ\gamma$ is absolutely continuous and
\[
ess\ell_{u,p} B(\gamma,\delta)\le \ell(u\circ\gamma)
\]
for every $\delta>0$.
\end{definition}

\begin{proposition}\label{prop:aeupreg}
	$p$-almost every Lipschitz curve is $(u,p)$-regular.
\end{proposition}
For the proof we will denote $\Gamma_{u,p}(\lambda)=\{\gamma: \ell(u\circ\gamma)<\lambda\}$.

\begin{proof}[Proof of Proposition \ref{prop:aeupreg}]
Denote by $\Gamma_0$ the set of curves in $\overline{\Gamma}(X)$ which are not $(u,p)$-regular. We may write $\Gamma_0=\Gamma_1\cup\Gamma_2$, where 
\[
\Gamma_1=\{ \gamma: u\circ\gamma\textrm{ not absolutely continuous} \},\quad \Gamma_2=\Gamma_0\setminus\Gamma_1.
\]
By the fact that $u\in\Neml pXY$ we have $\Mod_p(\Gamma_1)=0$. It remains to show that $\Mod_p(\Gamma_2)=0$. For any $\delta>0$ and $\gamma\in \overline\Gamma(X)$, set
\begin{align*}
\varepsilon(\delta,\gamma)&:=ess\ell_{u,p}B(\gamma,\delta)-\ell(u\circ\gamma)\\
\delta(\gamma)&:=\sup\{ \delta>0: \varepsilon(\delta,\gamma)>0 \}.
\end{align*}
Note that $\delta(\gamma)>0$ if and only if $\gamma\in\Gamma_2$. Using Remark \ref{rmk:essinf} we make the following observation which holds for $\delta ,\varepsilon>0$:
\begin{align}\label{eq:observation}
\Mod_p(B(\gamma,\delta)\cap \Gamma_{u,p}(\ell(u\circ\gamma)+\varepsilon))=0\textrm{ implies }\delta\le \delta(\gamma)\textrm{ and }\varepsilon\le \varepsilon(\delta,\gamma).
\end{align}
Moreover, for any $\gamma\in \Gamma_2$ and $\delta>0$, we have
\begin{equation}\label{eq:null}
\Mod_p(B(\gamma,\delta)\cap \Gamma_{u,p}(\ell(u\circ\gamma)+\varepsilon(\delta,\gamma)))=0.
\end{equation}

\bigskip\noindent Let $\{\gamma_i \}_{i\in\N}\subset \Gamma_2$ be a countable dense set. For each $i,k\in\N$ and rational $r>0$ let $\gamma_{i,k,r}\in B(\gamma_i,r)\cap \Gamma_2$ satisfy
\[
\ell(u\circ\gamma_{i,k,r})<\inf\{ \ell(u\circ\beta):\ \beta\in  B(\gamma_i,r)\cap \Gamma_2 \}+1/k.
\]
By (\ref{eq:null}) it suffices to prove that
\begin{equation*}
\Gamma_2\subset \bigcup_{i,k\in\N}\bigcup_{r,\delta\in \Q_+}B(\gamma_{i,k,r},\delta)\cap\Gamma_{u,p}(\ell(u\circ\gamma_{i,k,r})+\varepsilon(\delta,\gamma_{i,k,r})).
\end{equation*}

\bigskip\noindent For any $\gamma\in\Gamma_2$ let $i\in\N$ be such that $d_{\infty}(\gamma,\gamma_i)<\delta(\gamma)/8$. Choose rational numbers $r,\delta \in \Q_+$ such that $d_\infty(\gamma_i,\gamma)<r<\delta(\gamma)/8$ and $\delta(\gamma)/4<\delta<\delta(\gamma)/2$, and a natural number $k\in\N$ so that $1/k<\varepsilon(2\delta,\gamma)$.

\bigskip\noindent We will show that 
\begin{align}\label{eq:final}
\gamma\in B(\gamma_{i,k,r},\delta)\cap \Gamma_{u,p}(\ell(u\circ\gamma_{i,k,r})+\varepsilon(\delta,\gamma_{i,k,r})).
\end{align}
Indeed, since $\gamma_{i,k,r}\in B(\gamma_i,r)\cap\Gamma_2$, the triangle inequality yields
\[
d_{\infty}(\gamma,\gamma_{i,k,r})\le d_{\infty}(\gamma,\gamma_{i})+d_{\infty}(\gamma_i,\gamma_{i,k,r})<2r<\delta(\gamma)/4<\delta.
\]
In particular
\begin{equation}\label{1}
\gamma\in B(\gamma_{i,k,r},\delta)
\end{equation}
and also
\begin{equation}\label{2}
B(\gamma_{i,k,r},\delta)\subset B(\gamma,2\delta)
\end{equation}

To bound the length of $u\circ\gamma$, observe that
\begin{align*}
\ell(u\circ\gamma_{i,k,r})&<\inf\{ \ell(u\circ\beta):\ \beta\in B(\gamma_i,r)\cap\Gamma_2 \}+1/k<\ell(u\circ\gamma)+1/k\\
&<\ell(u\circ\gamma)+\varepsilon(2\delta,\gamma)
\end{align*}
Setting $\varepsilon:=\ell(u\circ\gamma)+1/k-\ell(u\circ\gamma_{i,k,r})>0$ it follows that
\begin{align}\label{3}
\Gamma_{u,p}(\ell(u\circ\gamma_{i,k,r})+\varepsilon)\subset \Gamma_{u,p}(\ell(u\circ\gamma)+\varepsilon(2\delta,\gamma))
\end{align}
Combining (\ref{2}) and (\ref{3}) with (\ref{eq:null}) we obtain 
\[
\Mod_p(B(\gamma_{i,k,r},\delta)\cap \Gamma_{u,p}(\ell(u\circ\gamma_{i,k,r})+\varepsilon))=0,
\]
which, by (\ref{eq:observation}) yields $\varepsilon\le \varepsilon(\delta,\gamma_{i,k,r})$. Thus
\[
\ell(u\circ\gamma)<\ell(u\circ\gamma)+1/k=\ell(u\circ\gamma_{i,k,r})+\varepsilon\le \ell(u\circ\gamma_{i,k,r})+\varepsilon(\delta,\gamma_{i,k,r})
\]
which, together with (\ref{1}), implies (\ref{eq:final}). This completes the proof.
\end{proof}

\bigskip\noindent For the next two results we assume that $u\in \Neml pXY$ is continuous. We record the following straightforward consequence of the definition of $d_{u,p}'$ and the continuity of $u$ as a lemma.
\begin{lemma}\label{lem:elementary}
For each $x,y\in X$ we have
\[
d(u(x),u(y))\le d'_{u,p}(x,y)
\]
\end{lemma}

For the next proposition, we denote by $\ell_{u,p}(\gamma)$ the length of a curve $\gamma$ with respect to the pseudometric $d_{u,p}$. 
\begin{proposition}\label{prop:upreg}
If $\gamma:[a,b]\to X$ is $(u,p)$-regular and $a\le t\le s\le b$ then $\gamma|_{[t,s]}$ is $(u,p)$-regular and
	\[
	\ell_{u,p}(\gamma)=\ell(u\circ\gamma).
	\]
\end{proposition}
\begin{proof}
Denote $\eta:=\gamma|_{[s,t]}$. Let \[
\Gamma_n:=\left\{c|_{[s,t]}\ :\ c \in B\left(\gamma,\frac{1}{n}\right)\cap \Gamma_{u,p}\left(\ell(u\circ\gamma)+\frac{1}{n}\right)\right\}.
\] Since $\gamma$ is $(u,p)$-regular we have
\[
\Mod_p(\Gamma_n)\geq \Mod_p\left(B\left(\gamma,\frac{1}{n}\right)\cap \Gamma_{u,p}\left(\ell(u\circ\gamma)+\frac{1}{n}\right)\right)>0.
\]
We claim that for every $\varepsilon>0$ there exists $n_0\in\N$ so that 
\begin{equation}\label{eq:ineq}
\Gamma_n\subset B\left(\eta,\frac{1}{n}\right)\cap \Gamma_{u,p}(\ell(u\circ\eta)+\varepsilon)
\end{equation}
for all $n\ge n_0$. Indeed, otherwise there exists $\varepsilon_0>0$ and a sequence $$\gamma_{n_k}\in B\left(\gamma,\frac{1}{n_k}\right)\cap \Gamma_{u,p}\left(\ell(u\circ\gamma)+\frac{1}{n_k}\right)$$ so that
\[
\ell(u\circ\gamma_{n_k}|_{[s,t]})\ge \ell(u\circ\eta)+\varepsilon_0.
\]
Thus 
\begin{align*}
\ell(u \circ \gamma)+\frac{1}{n_k}&\geq \ell(u \circ \gamma_{n_k})=\ell\left( u \circ \gamma_{n_k}|_{[s,t]}\right)+\ell\left(u \circ \gamma_{n_k}|_{[s,t]^c}\right)\\
&\geq \ell\left( u \circ \eta\right)+\varepsilon_0+\ell\left(u \circ \gamma_{n_k}|_{[s,t]^c}\right)
\end{align*}
yielding
\begin{equation}
\label{eqy.2}
\ell\left(u \circ \gamma_{|[s,t]^c}\right)+\frac{1}{n_k}\geq \ell\left(u \circ \gamma_{n_k}|_{[s,t]^c}\right)+\varepsilon_0.
\end{equation}
By taking $\liminf_{k\to\infty}$ in \eqref{eqy.2} we obtain
\begin{equation*}
\ell\left(u \circ \gamma_{|[s,t]^c}\right)\geq \ell\left(u \circ \gamma_{|[s,t]^c}\right)+\varepsilon_0,
\end{equation*}
which is a contradiction. 

Thus (\ref{eq:ineq}) holds true. If $\varepsilon,\delta>0$, let $n\in\N$ be such that (\ref{eq:ineq}) holds and $\delta>1/n$. We have
\[
0<\Mod_p\Gamma_n\le \Mod_p(B(\eta,\delta)\cap\Gamma_{u,p}(\ell(u\circ\eta)+\varepsilon)),
\]
implying $ess\ell_{u,p}B(\eta,\delta)\le \ell(u\circ\eta)+\varepsilon$. Since $\varepsilon>0$ is arbitrary, $\eta$ is $(u,p)$-regular.

\bigskip\noindent To prove the equality in the claim note that, since $\gamma$ is $(u,p)$-regular we have
\[
d_{u,p}(\gamma(t),\gamma(s))\le d_{u,p}'(\gamma(t),\gamma(s))\le \lim_{\delta\to 0}ess\ell_{u,p}B(\gamma|_{[s,t]},\delta)\le \ell(u\circ\gamma|_{[s,t]})
\]
for any $s\le t$. It follows that
\[
\ell_{u,p}(\gamma)\le \ell(u\circ\gamma).
\]

On the other hand Lemma \ref{lem:elementary} implies that
\[
d(u(\gamma(t)),u(\gamma(s)))\le d_{u,p}(\gamma(t),\gamma(s))
\]
whenever $s\le t$, from which the opposite inequality readily follows.
\end{proof}

\subsection{The Sobolev-to-Lipschitz property in thick quasiconvex spaces}\label{sec:sobtolip}

In this subsection let $p\in[1,\infty]$ and $X=(X,d,\mu)$ be $p$-thick quasiconvex with constant $C\ge 1$. Consider the map $u=\id\in \Neml pXX$. We denote by $d_p$ the pseudometric $d_{u,p}$ associated to~$u$.
\begin{lemma}
\label{lem:biLip}
	$d_p$ is a metric on $X$, and satisfies $d\le d_p\le Cd$.
\end{lemma}
\begin{proof}
Let $x,y\in X$ be distinct and $\delta>0$. Since
\[
\Mod_p\Gamma(B(x,\delta),B(y,\delta);C)>0
\]
it follows that
\[
ess\ell_{\id,p}\Gamma(B(x,\delta),B(y,\delta))\le C(d(x,y)+2\delta),
\]
implying $d_{\id,p}'(x,y)\le Cd(x,y)$. Thus $d_p\le Cd$ and in particular $d_p$ is finite-valued. Lemma \ref{lem:elementary} implies that $d(x,y)\le d_p(x,y)$ for every $x,y\in X$. These estimates together prove the claim.
\end{proof}

We are now ready to prove Theorem \ref{thm:minmetric}. For the proof, we denote by $X_p$ the space $(X,d_p,\mu)$  and by $B_p(x,r)$ balls in $X$ with respect to the metric $d_p$; $\ell_p$ and $\Mod_{X_p,p}$ refer to the length of curves and $p$-modulus taken with respect to $X_p$. 

\begin{proof}[Proof of Theorem \ref{thm:minmetric}]
By Lemma~\ref{lem:biLip} it follows that $X_p$ is an infinitesimally doubling metric measure space and $\Mod_{X_p,p}\Gamma =0$ if and only if $\Mod_p\Gamma =0$. In particular if $\Gamma_*$ denotes the set of curves $\gamma$ in $X$ such that $\ell(\gamma)\neq \ell_p(\gamma)$ then by Propositions~\ref{prop:aeupreg} and \ref{prop:upreg}
\begin{equation}
\label{eq:eqmod}
\Mod_{X_p,p}\Gamma_*= \Mod_p\Gamma_*=0.
\end{equation}
From \eqref{eq:eqmod} and the definition of modulus it follows that $\Mod_{X_p,p}\Gamma = \Mod_p\Gamma$ for every family of curves $\Gamma$ in $X$.\\
By \eqref{eq:eqmod} and Proposition~\ref{prop:sobtoClip} every $f\in  \Ne p{X_p}$ with $g_f\le 1$ has a $C$-Lipschitz representative.
If $x,y\in X$ are distinct and $\delta, \varepsilon >0$ note that the curve family
\[
\Gamma_1=\{ \gamma\in \Gamma(B(x,\delta),B(y,\delta)): \ell(\gamma)\le ess\ell_{\id,p}\Gamma(B(x,\delta),B(y,\delta))+\varepsilon \}
\]
satisfies
\[
\Mod_p\Gamma_1>0
\]
by Remark \ref{rmk:essinf} and the fact that $\varepsilon >0$. Note also that $$\ell(\gamma)\le ess\ell_{\id,p}\Gamma(B(x,\delta),B(y,\delta))+\varepsilon \le d_{\id,p}'(\gamma(1),\gamma(0))+\varepsilon$$ for $\gamma\in \Gamma_1$. Let $f\in \Ne p{X_p}$ satisfy $g_f\le 1$ almost everywhere. Let $\bar f$ be the Lipschitz representative of $f$, and $\Gamma_0$ a curve family with $\Mod_{p}\Gamma_0=0$ and
\[
|f(\gamma(1)-f(\gamma(0))|\le \ell_p(\gamma)=\ell(\gamma)
\]
whenever $\gamma\notin \Gamma_0$. We have
\[
\Mod_p(\Gamma_1\setminus \Gamma_0)>0,
\]
so that there exists $\gamma_\delta\in \Gamma_1\setminus \Gamma_0$. We obtain
\[
|\bar f(\gamma_\delta(1))-\bar f(\gamma_\delta(0))|\le \ell_p(\gamma_\delta)\le d_{\id,p}'(\gamma_\delta(1),\gamma_\delta(0))+\varepsilon.
\]
Letting $\delta\to 0$ yields $|\bar f(x)-\bar f(y)|\le d_{\id,p}'(x,y)+\varepsilon$. Since $x,y\in X$ and $\varepsilon>0$ are arbitrary it follows that $\bar f$ is 1-Lipschitz with respect to $d_p$. In particular $X_p$ has the $p$-Sobolev-to-Lipschitz property and hence Theorem~\ref{thm:sobtolip} implies it is $p$-thick geodesic. We prove the minimality in Proposition~\ref{prop:minimality}.
\end{proof}
The metric $d_p$ is the \emph{minimal} metric above $d$ which has the $p$-Sobolev-to-Lipschitz property. Proposition \ref{prop:minimality} provides a more general statement, from which the minimality discussed in the introduction immediately follows.
\begin{proposition}\label{prop:minimality}
Let $Y$ be a $p$-thick geodesic metric measure space and $f:Y\to X$ be a volume preserving $1$-Lipschitz map. Then the map
\[
f_p:=f:Y\to (X,d_p)
\]
is volume preserving and $1$-Lipschitz.
\end{proposition}
\begin{proof}
It suffices to show that
\[
d_{\id,p}'(f(x),f(y))\le d(x,y),\quad x,y\in Y.
\]
For any $A>1$, the curve family
\[
\Gamma:=\{ \gamma\in\Gamma(B(x,\delta),B(y,\delta)): \ell(\gamma)\le Ad(\gamma(1),\gamma(0)) \}
\]
has positive $p$-modulus in $Y$. Since $f$ is volume preserving and 1-Lipschitz we have
\[
0<\Mod_{Y,p}\Gamma\le \Mod_{X,p}f\Gamma.
\]

If $\gamma\in \Gamma$ then $f\circ\gamma\in \Gamma(B(f(x),\delta),B(f(y),\delta))$ and, moreover
\[
\ell(f\circ\gamma)\le \ell(\gamma)\le Ad(\gamma(1),\gamma(0))\le Ad(x,y)+2A\delta.
\]
It follows that $ess\ell_{\id,p}\Gamma(B(f(x),\delta),B(f(y),\delta))\le Ad(x,y)+2A\delta$, and thus
\[
d_{\id,p}'(f(x),f(y))\le Ad(x,y)+2A\delta.
\]
Since $A>1$ and $\delta>0$ are arbitrary, the claim follows.
\end{proof}
We have the following immediate corollary.
\begin{corollary}\label{cor:minmetric}
	Assume $X$ is infinitesimally doubling and $p$-thick geodesic. Then $d=d_p$.
\end{corollary}
Before considering essential pull-back metrics by non-trivial maps, we prove Proposition~\ref{prop:nages}. The proof is based on the fact that spaces with Poincar\'e inequality are thick quasiconvex, and on the independence of the minimal weak upper gradient on the exponent.
\begin{proposition}[\cite{dur12}]\label{prop:thickquasi}
	Let $p\in [1,\infty]$ and $X$ be a doubling metric measure space satisfying a $p$-Poincar\'e inequality. There is a constant $C\ge 1$ depending only on the data of $X$ so that $X$ is $p$-thick quasiconvex with constant $C$.
\end{proposition}
The inverse implication in Proposition~\ref{prop:thickquasi} only holds if $p=\infty$, see \cite{dur12,dur12'}.
\begin{proof}[Proof of Proposition \ref{prop:nages}]
	Assume $X$ is a doubling metric measure space supporting a $p$-Poincar\'e inequality, and $q\ge p\ge 1$. Since $\Mod_q\Gamma=0$ implies $\Mod_p\Gamma=0$, see \cite[Proposition 2.45]{bjo11}, we have
	\[
	d\le d_q\le d_p\le Cd
	\]
	for some constant $C$ depending only on $p$ and the data of $X$.
	
	Fix $x_0\in X$ and consider the function
	\[
	f:X\to \R,\quad x\mapsto d_p(x_0,x).
	\]
Then $f$ is Lipschitz and, by Propositions~\ref{prop:aeupreg} and~\ref{prop:upreg}, it has 1 as a $p$-weak upper gradient. Since $X$ is doubling and supports a $p$-Poincar\'e inequality, \cite[Corollary A.8]{bjo11} implies that the minimal $q$-weak and $p$-weak upper gradients of $f$ agree almost everywhere, and thus 1 is a $q$-weak upper gradient of $f$, i.e.
	\[
	|f(\gamma(1))-f(\gamma(0))|\le \ell(\gamma)\le \ell_q(\gamma)
	\]
	for $q$-almost every curve. The space $X_q$ has the $q$-Sobolev-to-Lipschitz property, see Theorem \ref{thm:sobtolip}, and the 1-Lipschitz representative of $f$ agrees with $f$ everywhere, since $f$ is continuous. By this and the definition of $d_q$ we obtain
	\[
	d_p(x_0,x)\le d_q(x_0,x),\quad x\in X.
	\]
	Since $x_0\in X$ is arbitrary the equality $d_p=d_q$ follows.
	
	For the remaining equality, note that $d\le d_\infty\le d_{ess}\le Cd$. Indeed, $X$ supports an $\infty$-Poincar\'e inequality, whence \cite[Theorem 3.1]{dur12'} implies the rightmost estimate with a constant $C$ depending only on the data of the $\infty$-Poincar\'e inequality. As above, the function
	\[
	g:X\to \R,\quad x\mapsto d_{ess}(x_0,x)
	\]
satisfies
	\[
	|g(\gamma(1))-g(\gamma(0))|\le \ell(\gamma)\le \ell_\infty(\gamma)
	\]
	for $\infty$-almost every $\gamma$, from which the inequality $d_{ess}\le d_\infty$ follows.
\end{proof}

Note that in the proof of Proposition~\ref{prop:thickquasi} we use that in $p$-Poincar\'e spaces the $q$-weak upper gradient does not depend on $q\geq p$. In general such an equality is not true, see \cite{diM15}, and we do not know whether we can weaken the assumptions in Proposition~\ref{prop:nages} from $p$-Poincar\'e inequality to $p$-thick quasiconvexity.

\subsection{Essential pull-back metrics by Sobolev maps}\label{sec:pullbackhomeo}
Throughout this subsection $(X,d,\mu)$ will be a doubling metric measure space satisfying \eqref{eq:relvolgrowth} with $Q\ge 1$ and supporting a weak $(1,Q)$-Poincar\'e inequality, and $Y=(Y,d)$ a proper metric space.

We will use the following observation without further mention. If $p>Q$, and $u:X\to Y$ has a $Q$-weak upper gradient in $L^p_\loc(X)$, then $u$ has a representative  $\bar u\in \Neml pXY$ and the minimal $p$-weak upper gradient of $\bar u$ coincides with the minimal $Q$-weak upper gradient of $u$ almost everywhere. See \cite[Chapter 2.9 and Appendix A]{bjo11} and \cite[Chapter 13.5]{HKST07} for more details.

The next proposition states that higher regularity of a map is enough to guarantee that the essential pull-back distance in Definition \ref{def:pullback} is a finite-valued pseudometric.

\begin{proposition}\label{prop:Xu}
	Let $p>Q$, and suppose $u\in \Neml pXY$.  The pull-back distance $d_u:=d_{u,Q}'$ in Definition \ref{def:pullback} is a pseudometric satisfying
	\begin{equation}\label{eq:duestimate}
d_u(x,y)\le C\diam B^{Q/p}d(x,y)^{1-Q/p}\left(\dashint_{\sigma B}g_u^p\ud\mu\right)^{1/p}
	\end{equation}
	whenever $B\subset X$ is a ball and $x,y\in B$, where the constant $C$ depends only on $p$ and the data of $X$. 
\end{proposition}
For the proof of Proposition \ref{prop:Xu} we define the following auxiliary functions. Let $x\in X$, $\delta>0$ and $\Gamma_0$ a curve family with $\Mod_Q\Gamma_0=0$. Set $f:=f_{x,\delta,\Gamma_0}:X\to \R$ by
\begin{align*}
f(y)=\inf\{ \ell(u\circ \gamma):\gamma\in \Gamma(\bar B(x,\delta),y)\setminus \Gamma_0 \},\quad y\in X.
\end{align*}
When  $\rho\in L^Q(X)$ is a non-negative Borel function, whence $\Mod_Q(\Gamma_\rho)=0$, we denote $f_{x,\delta,\rho}:=f_{x,\delta,\Gamma_\rho}$.

\begin{lemma}\label{lem:distfinite}
	Let $x,\delta$ and $\rho$ be as above. The function $f=f_{x,\delta,\rho}$ is finite $\mu$-almost everywhere and has a representative in $\Nel pX$ with $p$-weak upper gradient~$g_u$. The continuous representative $\bar f$ of $f$ satisfies
	\begin{equation}\label{eq:morreyest}
	|\bar f(y)-\bar f(z)|\le C(\diam B)^{Q/p}d(y,z)^{1-Q/p}\left(\dashint_{\sigma B}g_u^p\ud\mu\right)^{1/p}.
	\end{equation}
\end{lemma}
\begin{proof}
	Let $g\in L^p_{\loc}(X)$ be a genuine upper gradient of $u$ and let $\varepsilon>0$ be arbitrary. 
	We fix a large ball $B\subset X$ containing $\bar B(x,\delta)$ and note that there exists $x_0\in \bar B(x,\delta)$ for which $\M_B(g+\rho)^Q(x_0)<\infty$, since $(g+\rho)|_B\in L^Q(B)$; cf. \eqref{eq:maximalfunctionbound}. Arguing as in \cite[Lemma 4.6]{teri2} we have that
	\[
	f(y)\le \inf\left\{ \int_\gamma (g+\rho): \gamma\in \Gamma_{x_0y}\setminus \Gamma_\rho \right\}<\infty
	\]
	for almost every $y\in B$. 
	
	Let $\gamma\notin \Gamma_\rho$ be a curve such that $f(\gamma(1)),f(\gamma(0))<\infty$ and $\int_\gamma g<\infty$. We may assume that $|f(\gamma(1))-f(\gamma(0))|=f(\gamma(1))-f(\gamma(0))\ge 0$. If $\beta$ is an element of $\bar \Gamma(B(x,\delta),\gamma(0))\setminus\Gamma_\rho$ such that $\ell(u\circ\beta)<f(\gamma(0))+\varepsilon$ then the concatenation $\gamma\beta$ satisfies $\gamma\beta\in \Gamma(\bar B(x,\delta),\gamma(1))\setminus\Gamma_\rho$. Thus
	\begin{align*}
	|f(\gamma(1))-f(\gamma(0))|\le \ell(u\circ\gamma\beta)-\ell(u\circ\beta)+\varepsilon=\ell(u\circ\gamma)+\varepsilon\le \int_\gamma g+\varepsilon.
	\end{align*}
	It follows that $g\in L^p_\loc(X)$ is a $Q$-weak upper gradient for $f$. By \cite[Corollary 1.70]{bjo11} we have that $f(y)<\infty$ for $Q$-quasievery $y\in B$, and $f\in \Nel QX$; cf. \cite[Theorem 9.3.4]{HKST07}.

Moreover, since $g\in L^p_\loc(X)$, $f$ has a continuous representative $\bar f\in \Nel pX$ which satisfies (\ref{eq:morreyest}), cf. \cite[Theorem 9.2.14]{HKST07}.
\end{proof}

\begin{proof}[Proof of Proposition \ref{prop:Xu}]
To prove the triangle inequality, let $x,y,z\in X$ be distinct. Take $\delta>0$ small, $\rho\in L^Q(X)$ non-negative, and let $E\subset X$ be a set of $Q$-capacity zero such that $f_{x,\delta,\Gamma_\rho}$ and $f_{y,\delta,\Gamma_\rho}$ agree with their continuous representatives outside $E$. Remember that $\alpha\beta\notin\Gamma_\rho$ whenever $\alpha,\beta\notin\Gamma_\rho$. We have that
	\begin{align*}
	d'_{u,Q}(x,z)+d'_{u,Q}(z,y)&\ge \inf_{w\in \bar B(z,\delta)}f_{x,\delta,\Gamma_\rho\cup\Gamma_E}(w)+\inf_{v\in \bar B(z,\delta)}f_{y,\delta,\Gamma_\rho\cup\Gamma_E}(v)\\
	&\ge \inf_{w,v\in \bar B(z,\delta)\setminus E}[f_{x,\delta,\rho}(w)+f_{y,\delta,\rho}(v)].
	\end{align*}
	Together with the estimate \eqref{eq:morreyest} this yields
	\begin{align*}
	d'_{u,Q}(x,z)+d'_{u,Q}(z,y)&\ge \inf_{w\in \bar B(z,\delta)\setminus E}[f_{x,\delta,\rho}(w)+f_{y,\delta,\rho}(w)]-C\delta^{1-Q/p}\\
	&\ge \inf\{ \ell(u\circ\gamma): \gamma\in \Gamma(\bar B(x,\delta),\bar B(y,\delta))\setminus\Gamma_\rho \}- C\delta^{1-Q/p},
	\end{align*}
	where $C$ depends on $u,x$ and $y$ as well as the data. Since $\delta>0$ and $\rho$ are arbitrary it follows that
	\[
	d'_{u,Q}(x,z)+d'_{u,Q}(z,y)\ge d'_{u,Q}(x,y).
	\]
	Moreover, for any $y'\in \bar B(y,\delta)\setminus E$,
	\begin{align*}
	|f_{x,\delta,\rho}(y')|&\le C\diam B^{p/Q}d(x,y')^{1-Q/p}\left(\dashint_{\sigma B}g_u^p\ud\mu\right)^{1/p}\\
	&\le C\diam B^{p/Q}(d(x,y)+\delta)^{1-Q/p}\left(\dashint_{\sigma B}g_u^p\ud\mu\right)^{1/p}
	\end{align*}
	Taking supremum over $\rho$, and letting $\delta$ tend to zero, we obtain \eqref{eq:duestimate}.
\end{proof}

\begin{remark}
A slight variation of the proof of Proposition~\ref{prop:Xu} shows that also in the setting of Theorem~\ref{thm:minmetric} the essential pull-back distance defines a metric. In particular it satisfies the triangle inequality and one does not have to pass to the maximal semimetric below. In this case in the argument instead of the Morrey embedding one applies Proposition~\ref{prop:sobtoClip}.
\end{remark}
\bigskip\noindent Fix a continuous map  $u\in\Neml pXY$, where $p>Q$, and denote by $Y_u$ the set of equivalence classes $[x]$ of points $x\in X$, where $x$ and $y$ are set to be equivalent if $d_{u,Q}'(x,y)=0$. The pseudometric $d_{u,Q}'$ defines a metric $d_u$ on $Y_u$ by
\begin{align*}
d_u([x],[y]):=d_{u,Q}'(x,y),\quad x,y\in X.
\end{align*}
The natural projection map
\begin{align*}
\widehat{P}_u:X\to Y_u,\quad x\mapsto [x]
\end{align*}
is continuous by \eqref{eq:duestimate}. The map $u$ factors as $u=\widehat u\circ \widehat P_u$, where
\[
\widehat u:Y_u\to Y,\quad  [x]\mapsto u(x)
\]
is well-defined and 1-Lipschitz; cf. Lemma \ref{lem:elementary}. 

\begin{lemma}\label{thm:pullbacksobtolip}
Under the given assumptions we obtain the following properties.
\begin{itemize}
\item[(1)] $\widehat P_u\in \Neml pX{Y_u}$ and $g_{\widehat P_u}=g_u$ $\mu$-almost everywhere.
\item[(2)] If $u$ is proper then $Y_u$ is proper and the projection $\widehat P_u:X\to Y_u$ is proper and monotone.
\end{itemize}
\end{lemma}
Recall here that a map is called proper if the preimage of every compact set is compact or equivalently if the preimage of every singleton is compact.
\begin{proof}\emph{(1)} Suppose $\gamma$ is a $(u,Q)$-regular curve such that $g_u$ is an upper gradient of $u$ along $\gamma$. Then by Proposition~\ref{prop:upreg}
\begin{equation*}
d_u(\widehat{P}_u\circ \gamma(1),\widehat{P}_u\circ\gamma(0))\leq \ell(u\circ\gamma)\le \int_\gamma g_u.
\end{equation*}
By Proposition~\ref{prop:aeupreg} this implies that $\widehat P_u\in \Neml QX{Y_u}$ and $g_{\widehat P_u}\le g_u$. The opposite inequality follows because $\widehat{u}$ is $1$-Lipschitz.\par 
\emph{(2)} The factorization implies that $K\subset \widehat P_u( u\inv(\overline{\widehat u(K)}))$, for $K\subset Y_u$ and hence $Y_u$ is proper. Similarly $\widehat P_u^{-1}(K)\subset u^{-1}(\widehat{u}(K))$ and hence $\widehat{P}_u$ is proper. The proof of monotonicity is an adaptation of the proof of \cite[Lemma 6.3]{lyt18}.\par
Assume $\widehat P_u\inv(y)$ is not connected for some $y\in Y_u$. Then there are compact sets $K_1,K_2$ for which $\dist(K_1,K_2)>0$ and $\widehat P_u\inv(y)=K_1\cup K_2$. Let $S$ denote the closed and non-empty set of points $X$ whose distance to $K_1$ and $K_2$ agree. Let $a$ be the minimum of $x\mapsto \dist_u(y,\widehat{P}_u(x))$ on $S$. Note here that $\widehat{P}_u(S)$ is closed as $Y_u$ and $\widehat{P}_u$ are proper and hence the infimum is attained and positive.
Let $k_i\in K_i$ for $i=1,2$. Since $d_{u,Q}'(k_1,k_2)=0$, for every $\varepsilon>0$ and small enough $\delta>0$, Proposition \ref{prop:aeupreg} implies the existence of a $(u,Q)$-regular curve $\gamma\in \Gamma(B(k_1,\delta),B(k_2,\delta))$ with $\ell(u\circ\gamma)<\varepsilon$. If $\delta$ is chosen small enough the curve must intersect $S$ at some point $s:=\gamma(t)$. Since $\gamma$ is $(u,Q)$-regular, $\gamma|_{[0,t]}$ is $(u,Q)$-regular and it follows that \[
a\le d_{u,Q}'(s,k_1)\leq d_{u,Q}'(s,\gamma(0))+C\delta^{1-\frac{Q}{p}}\leq \ell(u\circ \gamma)+ C\delta^{1-\frac{Q}{p}}\leq \varepsilon+C\delta^{1-\frac{Q}{p}};
\]
cf. Proposition \ref{prop:upreg} and \eqref{eq:duestimate}. Choosing $\varepsilon>0$ and $\delta>0$ small enough this yields a contradiction. Thus $\widehat P_u\inv(y)$ is connected for every $y\in Y_u$.
\end{proof}
Assume additionally that $Y$ is endowed with a measure $\nu$ such that $Y=(Y,d,\nu)$ is a metric measure space. We equip $(Y_u,d_u)$ with the measure $\nu_u:=\widehat u^*\nu$, which is characterized by the property
\begin{equation}\label{eq:nuu}
\nu_u(E)=\int_Y\#(\widehat u\inv(y)\cap E)\ud\nu(y),\quad E\subset Y_u \textrm{ Borel};
\end{equation}
cf. \cite[Theorem 2.10.10]{fed69}. Note that in general  $\nu_u$ is not a $\sigma$-finite measure. For the next theorem, we say that a Borel map $u:X\to Y$ has Jacobian $Ju$, if there exists a Borel function $Ju:X\to[0,\infty]$ for which
\begin{equation}\label{eq:area}
\int_EJu\ud\mu=\int_Y\#(u\inv(y)\cap E)\ud\nu(y)
\end{equation}
holds for every Borel set $E\subset X$. The \emph{Jacobian} $Ju$, if it exists, is unique up to sets of $\mu$-measure zero.

\begin{theorem}\label{thm:pullbacksobtolip1}
Assume $u$ is proper, nonconstant and has a locally integrable Jacobian $J u$ which satisfies
\begin{equation}\label{eq:qr}
	g_u^Q\le KJu
	\end{equation}
	$\mu$-almost everywhere for some $K\ge 1$. Then the following properties hold.
\begin{itemize}
\item[(1)] $Y_u$ is a metric measure space and has the $Q$-Sobolev-to-Lipschitz property.
\item[(2)] $J u$ is the locally integrable Jacobian of $\widehat{P}_u$ and  $\#\widehat P_u\inv(y)=1$ for $\nu_u$-almost every $y\in Y_u$.
\end{itemize}
\end{theorem}
\begin{proof}
Note that
\begin{align*}
\int_{\widehat P_u\inv E}Ju\ud\mu=\int_Y\#(u\inv(y)\cap \widehat P_u\inv E)\ud\nu(y)\ge \int_Y\#(\widehat u\inv(y)\cap E)\ud\nu(y)=\nu_u(E)
\end{align*}
for every Borel set $E\subset Y_u$. Thus $\nu_u$ is a locally finite measure, and the estimate above implies that $\widehat P_u$ has a locally integrable Jacobian $J\widehat P_u\le Ju$. For any Borel set $E\subset X$ we have
\begin{align*}
\int_EJ \widehat P_u\ud\mu&=\int_{Y_u} \#(\widehat P_u\inv(y)\cap E)\ud\nu_u(y)=\int_Y\left(\sum_{y\in \widehat u\inv(z)}\#(\widehat P_u\inv(y)\cap E)\right)\ud\nu(z)\\
&=\int_Y\#(u\inv(z)\cap E)\ud\nu(z)=\int_EJu\ud\mu,
\end{align*}
which yields $J\widehat{P}_u=J u$ almost everywhere. Since $\widehat P_u$ is monotone and satisfies \eqref{eq:area} we have that $\#P_u\inv(y)=1$ for $\nu_u$-almost every $y\in Y_u$.\par
By Lemma~\ref{thm:pullbacksobtolip} to see that $Y_u$ is a metric measure space it remains to show that balls in $Y_u$ have positive measure. The idea of of proof is borrowed from~\cite[Lemma~6.11]{lyt18}. Assume $B=B(z,r)$ is a ball in $Y_u$ such that $\nu_u(B)=0$. Then we would have $J \widehat P_u=J_u=0$ almost everywhere on $U:=\widehat{P}_u^{-1}(B)$. By \eqref{eq:qr} and Lemma~\ref{thm:pullbacksobtolip} $g_{\widehat{P}_u}=g_u=0$ almost everywhere on~$U$. Thus $\widehat{P}_u$ is locally constant on~$U$. But $U$ is connected and hence $\widehat{P}_u$ is constant.\par 
To see that $U$ is connected let $x\in U$ satisfy $\widehat{P}_u(x)=z$. Then, for $y\in U$ and $\delta>0$ arbitrary, there is a $(u,Q)$-regular curve $\gamma\in \Gamma(B(x,\delta),B(y,\delta))$ such that $\ell(u\circ \gamma)<r$. For $\delta$ small enough by Proposition~\ref{prop:upreg} the image of $\gamma$ is contained in $U$. As this holds for all small enough $\delta$, the points $x$ and $y$ must lie in the same connected component of the open set $U$. Since $y$ was arbitrary, $U$ must be connected.\par 
We have deduced that $B=\widehat{P}_u(U)$ consists only of the single point $z$. So either $Y_u$ is disconnected or consists of a single point. The former is impossible because $X$ is connected and the later by the assumption that~$u$ is nonconstant.\par 
Before showing the Sobolev-to-Lipschitz property we first claim
\begin{align}\label{eq:modulus}
\Mod_Q\Gamma\le K\Mod_Q\widehat P_u(\Gamma)
\end{align}
for every curve family $\Gamma$ in $X$. Indeed, suppose $\rho\in L^Q(Y_u)$ is admissible for $\widehat P_u(\Gamma)$, and set $\rho_1=\rho\circ \widehat P_u$. Let $\Gamma_0$ a curve family with $\Mod_Q\Gamma_0=0$ for which $\gamma$ is $(\widehat P_u,Q)$- and $(u,Q)$-regular, and $g_u$ is an upper gradient of $\widehat P_u$ and $u$ along $\gamma$, whenever $\gamma\notin\Gamma_0$. For any $\gamma\in \Gamma\setminus\Gamma_0$ we have
\begin{align*}
1\le \int_{\widehat P_u\circ\gamma}\rho=\int_0^1\rho_1(\gamma(t))|(\widehat P_u\circ\gamma)_t'|\ud t\le \int_0^1\rho_1(\gamma(t))g_u(\gamma(t))|\gamma_t'|\ud t,
\end{align*}
i.e. $\rho_1g_u$ is admissible for $\Gamma\setminus\Gamma_0$. We obtain
\begin{align*}
\Mod_Q\Gamma=\Mod_Q(\Gamma\setminus\Gamma_0)&\le \int_X\rho_1^Qg_u^Q\ud\mu\le K\int_X\rho^Q\circ \widehat P_u Ju\ud\mu\\
&=K\int_{Y_u} \rho^Q\ud\nu_u.
\end{align*}
The last equality follows, since Lemma~\ref{thm:pullbacksobtolip} implies $\widehat{P}_{u*}(Ju\ \mu)=\nu_u$. Taking infimum over admissible $\rho$ yields \eqref{eq:modulus}.

We prove that $Y_u$ has the $Q$-Sobolev-to-Lipschitz property. Suppose $f\in \Ne Q{Y_u}$ satisfies $g_f\le 1$. There is a curve family $\Gamma_0$ in $Y_u$ with $\Mod_{Q}(\Gamma_0)=0$ such that
\begin{align*}
|f(\gamma(1))-f(\gamma(0))|\le \ell_u(\gamma)
\end{align*}
whenever $\gamma\notin\Gamma_0$. By \eqref{eq:modulus} we have
\[
\Mod_{Q}\widehat P_u\inv\Gamma_0\le K\Mod_{Q}\Gamma_0=0.
\]
Here $\widehat P_u\inv\Gamma$ denotes the family of curves $\gamma$ in $X$ such that $\widehat P_u\circ\gamma\in \Gamma_0$. Together with Propositions~\ref{prop:aeupreg} and~\ref{prop:upreg} this implies that, for $\Mod_{Q}$-almost every curve~$\gamma$ in~$X$, we have
\[
|f\circ \widehat P_u(\gamma(1))-f\circ \widehat P_u(\gamma(0))|\le \ell_u(\gamma)=\ell(u\circ\gamma)\le \int_\gamma g_u.
\]
Since $g_u\in L^p_{\loc}(X)$ for $p>Q$, it follows that $f\circ \widehat P_u$ has a continuous representative~$\bar f$.
Let $x,y\in X$ be distinct, $\varepsilon >0$ arbitrary, and $\delta>0$ such that \[
|\bar f(x)-\bar f(z)|+|\bar f(y)-\bar f(w)|<\varepsilon
\]
 whenever $z\in \bar B(x,\delta)$ and $w\in \bar B(y,\delta)$. Denote by $\Gamma_1$ the curve family with $\Mod_{X,Q}\Gamma_1=0$ so that
\[
|\bar f(\gamma(1))-\bar f(\gamma(0))|\le \ell(u\circ\gamma)\quad\textrm{whenever }\gamma\notin\Gamma_1.
\]
Then 
\begin{align*}
|\bar f(x)-\bar f(y)|\le \varepsilon + \underset{\gamma\in \Gamma(\bar B(x,\delta),\bar B(y,\delta))\setminus \Gamma_1}{\inf \ell(u\circ\gamma)}\le \varepsilon+d_{u,Q}'(x,y).
\end{align*}
Since $x,y\in X$ and $\varepsilon>0$ are arbitrary it follows that $\bar f(x)=\bar f(y)$ whenever $d_{u,Q}'(x,y)=0$ and that the map $[x]\mapsto \bar f(x)$ is 1-Lipschitz with respect to the metric $d_u$, and is a $\nu_u$-representative of $f$.
\end{proof}

\section{Essential metrics on surfaces}\label{sec:discs}
In this section we apply the constructions from Section \ref{sec:esspullback} to prove Theorems~\ref{thm:discmain} and~\ref{cor:plateau}.
\subsection{Ahlfors regular spheres}
\label{sec:spheres}
In what follows, $Z$ is an Ahlfors $2$-regular metric sphere which is linearly locally connected. We consider $Z$ as endowed with the Hausdorff $2$-measure~$\mathcal{H}^2_Z$. Denote by $\Lambda(Z)$ the set of cell-like maps $u\in N^{1,2}(S^2,Z)$. Recall that a continuous map is called cell-like if the preimage of every point is cell-like. The definition of general cell-like sets may be found in~\cite{lyt18} but for our purposes it suffices to know that a compact subset~$K$ of $S^2$ is cell-like if and only if $K$ and $S^2\setminus K$ are connected. Let us recall the following result from~\cite{lyt17}.
\begin{theorem}[\cite{lyt17}]
The set $\Lambda(Z)$ is nonempty and contains an element $u\in \Lambda(Z)$ of least Reshetnyak energy $E^2_+(u)$. Such $u$ is a quasisymmetric homeomorphism and unique up to conformal diffeomorhism of $S^2$.
\end{theorem}
In the following we fix such energy minimzing quasisymmetric homeomorphism $u\in N^{1,2}(S^2,Z)$. Then $Z$ supports a 2-Poincar\'e inequality, and $u$ and $u\inv$ are quasiconformal; cf. \cite[Corollary 8.15 and Theorem 9.8]{hei00}. Moreover, by \cite[Theorem~2.9 and Corollary~4.20]{kor09}, we have that $u\in \Nem p{S^2}Z$ for some $p>2$.\par 
We denote by $\widehat{Z}=(Z,d_2)$ the metric space as constructed in Section~\ref{sec:sobtolip} and by~$\widehat{Z}_u=(S^2,d_{u,2})$ the metric space as constructed in Section~\ref{sec:pullbackhomeo}. It turns out that in fact the two constructions coincide.
\begin{lemma}
 The map $\widehat{Z}_u\rightarrow \widehat{Z}$ induced by $u$ defines an isometry.
\end{lemma}
\begin{proof}
Since $u$ is a homeomorphism, the claim follows immediately from the construction and the fact that
\begin{equation*}
\Mod_2(\Gamma)>0 \Leftrightarrow \Mod_2(u\circ \Gamma)>0
\end{equation*}
for every family of curves $\Gamma$ in $S^2$.
\end{proof}
In the following we identify $\widehat{Z}$ and $\widehat{Z}_u$ under this isometry. Thus we have a canonical factorization $u=\widehat{u}\circ \widehat{P}_u$ where $\widehat{P}_u:S^2\rightarrow \widehat{Z}$ is a homeomorphism and $\widehat{u}:\widehat{Z}\rightarrow Z$ a $1$-Lipschitz homeomorphism.\par
Note that $Z$ as supports a $2$-Poincar\'e in particular it is $2$-thick quasiconvex. By Theorem~\ref{thm:minmetric} the space $\widehat{Z}$ is bi-Lipschitz equivalent to $Z$ and hence itself an Ahlfors $2$-regular metric sphere which is linearly locally connected.\par
Another crucial observation is the following. If $Y$ is another complete metric space and $v\in N^{1,2}(S^2,Y)$, then \begin{equation*}
\apmd u=\apmd v
\end{equation*} holds almost everywhere on $S^2$ if and only if
  \begin{equation*}
 \ell(u\circ \gamma)=\ell(v\circ \gamma)
 \end{equation*}
holds for $2$-almost every curve $\gamma$ in $S^2$. In particular $Z_u=Y_v$ if either of the two conditions holds, see \cite[Lemma 3.1 and Corollary 3.2]{lyt18}.

\begin{proof}[Proof of Theorem \ref{thm:discmain}]
By Theorem~\ref{thm:pullbacksobtolip} we have that $\widehat{P}_u\in N^{1,p}(S^2,\widehat{Z})$. By Proposition~\ref{prop:aeupreg} and Proposition~\ref{prop:upreg}, $2$-almost every curve $\gamma$ in $S^2$ satisfies
$\ell(u\circ \gamma)=\ell(\widehat{P}_u\circ \gamma)$. In particular
\begin{equation}
\label{eq:eqmetdif}
\apmd u =\apmd \widehat{P}_u
\end{equation}
 holds almost everywhere on $S^2$. Assume $v\in \Lambda(\widehat{Z})$ is such that $E^2_+(v)<E^2_+(\widehat{P}_u)$. Then $\widehat{u}\circ v \in \Lambda(Z)$ and
\begin{equation}
E^2_+(\widehat{u}\circ v)\leq E^2_+(v)<E^2_+(\widehat{P}_u)=E^2_+(u).
\end{equation}
This contradicts the assumption that $u$ is of minimal energy and hence~$\widehat{P}_u$ is energy minimizing in $\Lambda(\widehat{Z})$. By \eqref{eq:eqmetdif} we obtain that~$\widehat{Z}$ is analytically equivalent to~$Z$.\par
By Theorem~\ref{thm:minmetric} we know that $\widehat{Z}$ is thick geodesic~\textbf{with respect to $\mathcal{H}^2_Z$}. We want thick geodecity however to hold with respect to $\mathcal{H}^2_{\widehat{Z}}$. In particular to see property~(2) it suffices to show that \begin{equation}
\label{eq:measeq}
\mathcal{H}^2_Z=\mathcal{H}^2_{\widehat{Z}}.
\end{equation}
Equality \eqref{eq:measeq} follows however immediately from~\eqref{eq:eqmetdif}, Theorem~\ref{thm:rectarea} and the fact that $\widehat{P}_u$ and $\widehat{u}$ are homeomorphisms which satisfy higher integrability.\par 
Since $\widehat{Z}_u$ is Ahlfors $2$-regular, Theorem~\ref{thm:sobtolip} implies that $\widehat{Z}$ has the $2$-Sobolev-to-Lipschitz property. (We could derive this also directly from Theorem~\ref{thm:pullbacksobtolip1} by showing $\nu_u=\mathcal{H}^2_{\widehat{Z}}$ where $\nu=\mathcal{H}^2_Z$.)\par 
Let $Y$ be analytically equivalent to $\widehat{Z}$ and hence also to $Z$. Let $v\in \Lambda(Y)$ be the energy minimizing homeomorphism. Then $Y_v=\widehat{Z}_u=\widehat{Z}$ and hence the desired surjective $1$-Lipschitz homeomorphism $\widehat{Z}\rightarrow Y$ is given by $f=\widehat{v}$. This proves (3).\par 
Next we prove (4), i.e. volume rigidity. Let $Y$ be an Ahlfors $2$-regular metric sphere which is linearly locally connected. Suppose $f:Y\rightarrow \widehat{Z}$ is $1$-Lipschitz, cell-like and volume preserving with respect to $\mathcal{H}^2_Y$ and $\mathcal{H}^2_Z$. Let $v\in \Lambda(Y)$ be an energy minimizer. Then $v$ is a quasisymmetric homeomorphism, and
\begin{equation*}
\mathcal{H}^2(Y)=\mathcal{H}^2(\widehat{Z})\leq \int_{S^2} J(\apmd_p f\circ v)\leq  \int_{S^2} J(\apmd_p v)=\mathcal{H}^2(Y).
\end{equation*}
Thus $f\circ v$ is an area minimizer in $\Lambda(\widehat Z)$ and $\apmd f\circ v =\apmd v$ almost everywhere. Since $v$ is infinitesimally quasiconformal, see \cite[Theorem~6.6]{lyt17}, the same holds for~$f\circ v$. By the proof of~\cite[Theorem 6.3]{lyt17} this shows that $f\circ v$ is quasisymmetric. It follows that $f$ is quasisymmetric and thus $f\inv$ is quasiconformal. Proposition~\ref{prop:volrig} implies that $f$ is an isometry.\par 
	
	It remains to check that $\widehat{Z}$ is characterized uniquely by any of the listed properties. For the forthcoming discussion, we fix a metric sphere $Y$ which is analytically and bi-Lipschitz equivalent to $Z$. We moreover fix an energy minimal homeomorphism $v:S^2 \to Y$ such that $\apmd u=\apmd v$ almost everywhere.

	If $Y$ satisfies (1) (resp. (2)), then by Theorem \ref{thm:sobtolip} it satisfies (2) (resp. (1)) and, by Theorem \ref{thm:minmetric} (see Corollary \ref{cor:minmetric}), we have that $\widehat Y=Y$. But under the usual identifications $\widehat{Y}=\widehat{Y}_v=\widehat{Z}_u=\widehat{Z}$ and hence $Y$ is isometric to $\widehat{Z}$.
	
	If $Y$ satisfies (3), then there are surjective $1$-Lipschitz maps $f:Y\rightarrow \widehat{Z}$ and $g:\widehat{Z}\rightarrow Y$. Since $Y$ and $\widehat{Z}$ are compact, the composition $f\circ g$ is an isometry, see \cite[Theorem 1.6.15]{bur01}. By the fact that $f$ is $1$-Lipschitz it follows that $g$ is an isometry.
	
	Assume $Y$ satisfies (4). The canonical map $\widehat Y\to Y$ is 1-Lipschitz and volume preserving. In particular by volume rigidity $\widehat{Y}$ is isometric to $Y$. However the former is isometric to $\widehat{Z}$. This concludes the proof.
	\end{proof}

\subsection{Metric discs}
\label{sec:discss}
	In this subsection let $Z$ be a geodesic metric disc such that $(Z,\mathcal{H}^2_Z)$ is a metric measure space. 
	First we discuss the following weaker variant of Theorem~\ref{thm:discmain} in the disc setting.
\begin{theorem}\label{maxdiscgen}
 Assume $u\in \Lambda(\partial Z,Z)$ is a minimizer of $E^2_+$ which is moreover monotone and lies in~$N^{1,p}(\overline{\D},Z)$ for some $p>2$. Then there is a geodesic metric disc $\widehat{Z}_u$ and a factorization $u=\widehat{u}\circ \widehat{P}_u$ such that:
	\item[(1)] $\widehat{Z}_u$ satisfies the Sobolev-to-Lipschitz property.
	\item[(2)] $\widehat{P}_u\in \Lambda(\partial \widehat Z_u,\widehat Z_u)$ is a minimizer of $E^2_+$, contained in $N^{1,p}(\overline\D,\widehat{Z}_u)$, and a uniform limit of homeomorphisms.
	\item[(3)] For $2$-almost every curve $\gamma$ in $\overline\D$ one has $\ell(\widehat{P}_u\circ \gamma)=\ell(u\circ \gamma)$.
	\item[(4)] If $Y$ is a disc, and $v\in \Lambda(\partial Y,Y)$ satisfies $\apmd u=\apmd v$ almost everywhere, then there is a surjective $1$-Lipschitz map $j:\widehat{Z}\rightarrow Y$ such that $v=j\circ \widehat{P}_u$.
	\item[(5)] If $Z$ satisfies a $(C,l_0)$-quadratic isoperimetric inequality then $\widehat{Z}_u$ satisfies a $(C,l_0)$-quadratic isoperimetric inequality.
\end{theorem}
Note that the existence of such $u$ is guaranteed if $Z$ satisfies a quadratic isoperimetric inequality and $\partial Z$ satisfies a chord-arc condition, see~\cite{lyt17,lyt17''}.
\begin{proof}
Let $\widehat{Z}_u=(\overline{\D},d_u)$ be the space constructed as in Section~\ref{sec:pullbackhomeo} and $u=\widehat{u}\circ \widehat{P}_u$ be the associated factorization as discussed therein. Then as in the previous subsection we see that $\apmd u=\apmd \widehat{P}_u$ holds almost everywhere and hence in particular $\ell(u\circ \gamma)=\ell(\widehat{P}_u\circ \gamma)$ for $2$-almost every curve $\gamma$ in~$\overline{\D}$. As in the previous subsection this implies that $u$ and $\widehat{P}_u$ are infinitesimally quasiconformal.
	
Next we show that $\widehat Z_u$ is a metric disc, and that $\widehat P_u\in\Lambda(\partial \widehat Z_u,\widehat Z_u)$ is a uniform limit of homeomorphisms. By \cite[Corollary 7.12]{lyt18}, to see this it suffices to prove that $\widehat{P}_u$ is cell-like and its restriction $S^1\rightarrow \widehat{P}_u(S^1)$ is monotone. Since by Lemma~\ref{thm:pullbacksobtolip} $\widehat{P}_u$ is monotone, the proof that $\widehat{P}_u$ is cell-like follows by the same argument as in the proof of \cite[Theorem 8.1]{lyt18}. Thus it suffices to show that $\widehat P_u|_{S^1}:S^1\to \widehat P_u(S^1)$ is monotone, i.e that $\widehat{P}_u^{-1}(y) \cap S^1$ is connected for every $y\in \widehat{Z}_u$.
	
	We identify $\overline{\D}$ with the lower hemisphere of $S^2$. Since $\widehat{P}_u$ is cell-like, the set $K:=\widehat{P}_u^{-1}(y)$ is a cell-like subset of $S^2$ and in particular $K$ and $S^2\setminus K$ are connected. Assume $K\cap S^1$ is not connected. Then there exist $z,w \in S^1$ such that neither of the of the arcs connecting $z$ and $w$ is contained in $K$. However $u(z)=u(w)\in \partial Z$. Then since $u\in \Lambda(\partial Z,Z)$, one of the arcs $A \subset S^1$ connecting $z$ and $w$ gets mapped constantly to~$u(w)$. Set $M:=K\cup A$. Clearly $M$ is connected. We claim that $S^2\setminus M$ is also connected. Assume it was not, and let $O$ be a connected component of $S^2 \setminus M$ contained in $\overline{\D}$. Then as in the proof of \cite[Theorem 8.1]{lyt18}, the restriction of $u$ to $\partial O$ is identically constant $u(w)$, and hence $u|_{O}$ is identically constant $u(w)$. By the definition of $\widehat P_u$ it follows that $\widehat{P}_u$ is constant on $O$. As $\partial O \cap K \neq\emptyset$ we obtain $O\subset K$, which is a contradiction.\par 
	Thus $M$ is a cell-like subset of $S^2$. By Moore's quotient theorem, see e.g. \cite[Theorem~7.11]{lyt18}, $S^2/M$ is homeomorphic to $S^2$. Since $S^2/M$ is obtained from $S^2/K \cong S^2$ by quotienting out a closed curve that only self-intersects at one point, the arising space would have a topological cutpoint which is not the case for $S^2$. This contradiction shows that the restriction of $\widehat{P}_u$ to $S^1$ must be monotone and hence $\widehat{Z}_u$ is a metric disc, and $\widehat{P}_u\in \Lambda(\partial \widehat{Z}_u,\widehat{Z}_u)$ a uniform limit of homeomorphisms.\
	
	If $\widehat{P}_u$ is not an energy minimizer, then there exists $v\in \Lambda(\partial \widehat{Z}_u,\widehat{Z}_u)$ such that $E^2_+(v)<E^2_+(\widehat{P}_u)$. This implies that $\widehat{u}\circ v\in \Lambda(\partial Z,Z)$, and that
	\begin{equation}
	E^2_+(\widehat{u}\circ v)\leq E^2_+(v)<E^2_+(\widehat{P}_u)=E^2_+(u).
	\end{equation} 
	This contradiction shows that $\widehat{P}_u$ is an energy minimizer.
	
Set $\nu:=\mathcal{H}^2_Z$. By Theorem~\ref{thm:pullbacksobtolip1} we know that $\widehat{Z}_u$ has the Sobolev-to-Lipschitz property with respect to $\nu_u$. Recall that $u$ is infinitesimally quasiconformal and that, by Theorem~\ref{thm:rectarea} and monotonicity of $u$, the equality $J(\apmd u)=J u$ holds almost everywhere. We want however the Sobolev-to-Lipschitz property to hold with respect to $\mathcal{H}^2_{\widehat{Z}_u}$. To this end we show that $\nu_u=\mathcal{H}^2_{\widehat{Z}_u}$.
	
Since $u$ and $\widehat P_u$ are monotone and surjective, it follows that the map $\widehat u:\widehat Z_u\to Z$ is monotone. This and \eqref{eq:nuu} imply that $\#\widehat u\inv(z)=1$ for $\Ha^2_Z$-almost every $z\in Z$. Using the equality of the approximate metric differentials, we have
	\begin{align*}
	\nu_u(A)&=\int_{Z}\#(\widehat u\inv(z)\cap A)\ud\Ha^2_Z(z)=\Ha^2_Z(\widehat u(A))\\
	&=\int_{u\inv(\widehat u(A))}J(\apmd_x \widehat{P}_u)\ud x=\int\#(\widehat P_u\inv(y)\cap u\inv(\widehat u(A)))\ud \Ha^2_{\widehat{Z}_u}(y)\\
	&\ge \Ha^2_{\widehat{Z}_u}(A)
	\end{align*}
	for any Borel set $A\subset \widehat Z_u$. On the other hand
	\begin{equation}
		\nu_u(A)=\Ha^2_Z(\widehat u(A))\le \Ha^2_{\widehat{Z}_u}(A).
	\end{equation}
	These two inequalities imply that $\nu_u=\Ha^2_{\widehat{Z}_u}$. In particular we have shown that $\widehat{Z}_u$ has the Sobolev-to-Lipschitz property with respect to $\mathcal{H}^2_{\widehat{Z}_u}$.\par 
	The proof of the maximality statement (4) is identical to the proof of maximality in Theorem~\ref{thm:discmain}.\par 
	Finally assume that $Z$ satisfies a $(C,l_0)$-quadratic isoperimetric inequality. Note first that, by \cite[Proposition 5.1]{lyt18} and Theorem~\ref{thm:rectarea}, for every Jordan domain $V\subset \overline{\D}$ with $\ell(u|_{\partial V})<l_0$, one has
\begin{equation}
\textnormal{Area}\left(\widehat{P}_{u|V}\right)=\textnormal{Area}\left(u|_{V}\right)\leq C\cdot \ell(u|_{\partial V})^2\leq C\cdot \ell\left(\widehat{P}_{u|\partial V}\right)^2.
\end{equation}
The remaining proof is a variation of the proof of \cite[Theorem 8.2]{lyt18}. We say that points $p,q\in \overline{\D}$ are sufficiently connected if for every $\epsilon>0$ and there exists a curve $\eta$ joining $p$ to $q$ such that 
\begin{equation}\label{eq:1}
\ell(\widehat{P}_u\circ \eta)<\widehat{d}_u(p,q)+\epsilon.
\end{equation}
If we could show that all points $p,q \in \overline{\D}$ are sufficiently connected then the proof of \cite[Theorem 8.2]{lyt18} would go through without changes. Unfortunately we do not know whether this is true.\par 
Let $A$ be the set of those $p\in \overline{\D}$ such that $p$ either lies in the interior of $\overline{\D}$ or $p$ lies in an open interval $I\subset S^1$ satisfying  $\ell(\widehat{P}_u \circ I)<\infty$. We claim that if $p,q\in A$ then $p$ and $q$ are sufficiently connected.

So let $p,q\in A$. The proof of \cite[Corollary 5.4]{lyt18} shows that for every $\epsilon >0$ there exists $\delta >0$ such that, for every $y\in B_\delta(p)$ and $z\in B_\delta(q)$, there are curves $\gamma_p,\gamma_q$ connecting $y$ to $p$ and $z$ to $q$, respectively, such that \begin{equation}\label{eq:3}
\ell(\widehat{P}_u\circ \gamma_p)<\epsilon,\quad \ell(\widehat{P}_u\circ \gamma_q)<\epsilon.
\end{equation}
By construction of the metric $\widehat{d}_u$, there exists a curve $\gamma$ connecting $B_\delta(p)$ to $B_\delta(q)$ such that \[\ell(\widehat{P}_u\circ \gamma)<\widehat{d}_u(p,q)+\epsilon.\]
The concatenation of $\gamma$ with correspondingly chosen $\gamma_p$ and $\gamma_q$ yields a curve $\eta$ connecting $p$ to $q$ which satisfies \[\ell(\widehat{P}_u\circ eta)<\widehat{d}_u(p,q)+3\epsilon.\]
Since $\epsilon>0$ was arbitrary the claim follows.\par 
Now let $U\subset \widehat{Z}_u$ a Jordan domain satisfying $\ell(\partial U)<l_0$. Then the set of points in~$\partial U$ which have a preimage in~$A$ is dense in~$\partial U$. In particular we may choose the points $t_i$ and $x_i$ in the proof of \cite[Theorem 8.2]{lyt18} such that all the points $x_i$ lie in $A$ and are hence pairwise sufficiently connected. Under this modification the argument in the proof of \cite[Theorem 8.2]{lyt18} shows that
\begin{equation}
\mathcal{H}^2(U)\leq C\cdot \ell(\partial U)^2
\end{equation}
which completes the proof.
\end{proof}
\subsection{Essential minimal surfaces}
\label{sec:plat}
\begin{proof}[Proof of Theorem~\ref{cor:plateau}]

Let $X$ be a proper metric space which satisfies a $(C,l_0)$-quadratic isoperimetric inequality, and  $u\in\Lambda(\Gamma,X)$ a minimal disc spanning a Jordan curve $\Gamma$ which satisfies a chord-arc condition. By~\cite[Theorem 3.1]{lyt16} one has that $u\in N^{1,p}(\overline{\D};X)$ for some $p>2$. Let $\widehat{Z}_u$ be the metric space discussed in Section~\ref{sec:pullbackhomeo} and $u=\widehat{u}\circ \widehat{P}_u$ be the corresponding factorization.
Let $Z_u$ be the metric disc discussed in Section~\ref{sec:app} and $u=\bar{u}\circ P_u$ the corresponding factorization. Then, by \cite{lyt18} and the proof of \cite[Theorem 2.7]{cre19}, the pair $(Z_u,P_u)$ satisfies the assumptions of Theorem~\ref{maxdiscgen}. On the other hand, since $\ell(P_u \circ \gamma)=\ell(u\circ \gamma)$ for every curve~$\gamma$ in~$\overline{\D}$, it follows that $\widehat{Z}_u=\widehat{Z}_{\widehat{P}_u}$ and $\widehat{P}_u=\widehat{P}_{P_u}$. Theorem~\ref{maxdiscgen} now implies the claim.
\end{proof}
\bigskip\noindent The following example from \cite[Example 5.9]{lyt17}, demonstrates that a unique characterization statement in Theorem \ref{cor:plateau} in terms of the Sobolev-to-Lipschitz property does not hold.

\begin{example}\label{ex:segcollapsed}
	Let $Z$ be the metric space obtained from the standard Euclidean disc by collapsing a segment $I$ in its interior to a point. Then $Z$ is a geodesic metric disc satisfying a quadratic isoperimetric inequality with constant~$\frac{1}{2\pi}$, compare the proof of \cite[Theorem 3.2]{cre19}. The canonical quotient map $u:\overline{\D} \rightarrow Z$ is an energy minimizer in $\Lambda(\partial Z,Z)$ and satisfies $\apmd u=\apmd \id_{\overline{\D}}$ almost everywhere. In particular $\widehat{Z}_u$ is isometric to $\overline{\D}$. However $Z_u$ is isometric to $Z$, see~\cite[Theorem 1.2]{cre19}. It is also straightforward to see that $Z_u$ has the Sobolev-to-Lipschitz property.
\end{example}
To see that $Z$ is thick geodesic let $E\subset Z$ be a measurable subset of positive measure and $C>1$. Furthermore let $p\in \overline{\D}$ be such that $u(p)\neq u(I)$ is a density point of $E$ and $q\in I$ the point which is closest to $p$. Then for $\delta>0$ sufficiently small
\begin{equation}
\label{ineqex}
0<\Mod_2\Gamma(u^{-1}(E)\cap B(p,\delta),I\cap B(q,\delta);1)\le \Mod_2\Gamma(E,u(I);C),
\end{equation}
see \cite[Remark 3.4]{dur12} for the first inequality.  Now having equation~\eqref{ineqex} it is not hard to deduce that $Z$ is thick geodesic. The example also shows that being thick quasiconvex with constant~$1$ is a strictly stronger condition than being thick geodesic.\par
The metric disc $Z$ is not Ahlfors regular, since $\Ha^2(B_Z(p,r))$ grows linearly in $r$, and thus Example \ref{ex:segcollapsed} does not contradict Theorem \ref{thm:discmain}.
Note that, for $Z$ in the example, the construction in \cite{lyt18} yields the original space $Z$, while the space $\widehat Z$ constructed in the proof of Theorem \ref{maxdiscgen} coincides with $\overline\D$. This need not always be the case when collapsing a cell-like subset in the interior of $\overline{\D}$; the Euclidean disc with a small ball (in the  interior) collapsed is a metric disc satisfying the assumptions of Theorem \ref{maxdiscgen}, where both constructions yield the original space; cf. \cite[Example 11.3]{lyt18}.

\bibliographystyle{plain}
\bibliography{metric_discs_arxiv_v2}
\end{document}